\documentclass[journal,twoside,web]{ieeecolor}
\usepackage{generic}
\usepackage[utf8]{inputenc} 
\usepackage[T1]{fontenc} 
\usepackage{hyperref}       
\usepackage{url}            
\usepackage{booktabs}       
\usepackage[export]{adjustbox}
\usepackage{amsfonts,amsmath,amsthm}       
\usepackage{nicefrac}       
\usepackage{microtype}      
\usepackage{xcolor}         
\usepackage{graphicx}
\usepackage{subcaption}
\usepackage[shortlabels]{enumitem}
\usepackage{cite}
\usepackage[noend]{algorithmic}
\usepackage{cleveref}[2012/02/15]
\usepackage{epstopdf}
\usepackage[ruled]{algorithm}
\usepackage{caption}
\theoremstyle{plain}
\newtheorem{theorem}{Theorem}

\newtheorem{lemma}{Lemma}

\theoremstyle{definition}
\newtheorem{definition}{Definition}
\newtheorem{assumption}{Assumption}
\theoremstyle{remark}
\newtheorem{remark}{Remark}

\newcommand*{\mb}[1]{\mathbf{#1}}%
\newcommand*{\mc}[1]{\mathcal{#1}}%
\newcommand*{\mbb}[1]{\mathbb{#1}}%
\newcommand*{\bx}{\mathbf{x}}%
\newcommand*{\bz}{\mathbf{z}}%
\newcommand*{\f}{\mathbf{f}}%
\newcommand*{\h}{\mathbf{h}}%
\newcommand*{\N}{\mathbb{N}}%
\newcommand*{\R}{\mathbb{R}}%
\newcommand*{\prox}{\operatorname{prox}}%

\hypersetup{hidelinks=true}
\usepackage{textcomp}
\def\BibTeX{{\rm B\kern-.05em{\sc i\kern-.025em b}\kern-.08em
    T\kern-.1667em\lower.7ex\hbox{E}\kern-.125emX}}
\begin{document}
\title{Historical Information Accelerates Decentralized Optimization: A Proximal Bundle Method}
\author{Zhao Zhu, Yu-Ping Tian, and Xuyang Wu
\thanks{Z. Zhu and X. Wu are with the School of Automation and Intelligent Manufacturing, Southern University of Science and Technology, Shenzhen, China, and the State Key Laboratory of Autonomous
Intelligent Unmanned Systems, Beijing 100081, China. Email: {\tt\small 12532828@mail.sustech.edu.cn; wuxy6@sustech.edu.cn}}%
\thanks{Y.-P. Tian is with the School of Information and Electrical Engineering, Hangzhou City University, Hangzhou, China. Email: {\tt\small tianyp@hzcu.edu.cn}}
\thanks{This work is supported in part by the Guangdong Provincial Key Laboratory of Fully Actuated System Control Theory and Technology under grant No. 2024B1212010002, in part by the Shenzhen Science and Technology Program under grant No. JCYJ20241202125309014, and in part by the State Key Laboratory of Autonomous Intelligent Unmanned Systems under grant No. ZZKF2025-1-3.}}

\maketitle

\begin{abstract}

Historical information, such as past function values or gradients, has significant potential to enhance decentralized optimization methods for two key reasons: first, it provides richer information about the objective function, which also explains its established success in centralized optimization; second, unlike the second-order derivative or its alternatives, historical information has already been computed or communicated and requires no additional cost to acquire. Despite this potential, it remains underexploited. In this work, we employ a proximal bundle framework to incorporate the function values and gradients at historical iterates and adapt the framework to the proximal decentralized gradient descent method, resulting in a Decentralized Proximal Bundle Method (DPBM). To broaden its applicability, we further extend DPBM to the asynchronous and stochastic setting. We theoretically analysed the convergence of the proposed methods. Notably, both the asynchronous DPBM and its stochastic variant can converge with fixed step-sizes that are independent of delays, which is superior to the delay-dependent step-sizes required by most existing asynchronous optimization methods, as it is easier to determine and often leads to faster convergence. Numerical experiments on classification problems demonstrate that by using historical information, our methods yield faster convergence and stronger robustness in the step-sizes.
\end{abstract}

\section{Introduction}\label{sec:intro}
In many engineering scenarios, a network of nodes collaborates to find a common decision that minimizes the sum of all nodes' local costs. This problem, known as decentralized optimization, has found applications in diverse areas such as sensor networks, multi-robot coordination, power systems \cite{yang2016distributed}, and machine learning \cite{ lian2017can}.

To solve decentralized optimization problems efficiently, a large number of algorithms have been proposed \cite{lu2011gossip,nedic2009distributed,yuan2016convergence,zeng2018nonconvex,xu2021dp,nedic17,qu2017harnessing,qu2019accelerated,shi2015extra,koloskova2020unified,Wu25,zhou2025asynchronous,Zhu2024,zhang2019fully,tian2020achieving,ubl2021totally}. Early efforts primarily focus on consensus-based primal methods, where a notable example is the decentralized gradient descent (DGD) method \cite{nedic2009distributed,yuan2016convergence,zeng2018nonconvex}. Due to its simplicity and effectiveness, DGD has been extended to various settings, including stochastic optimization \cite{lian2017can}, privacy-aware learning \cite{xu2021dp}, and asynchronous optimization \cite{Wu25,zhou2025asynchronous}. However, it can only converge to an inexact optimum when using fixed step-sizes. To overcome this issue, a series of exactly convergent methods are proposed, such as EXTRA \cite{shi2015extra}, DIGing\cite{nedic17,qu2017harnessing}, and several primal-dual methods. Despite the rich literature on decentralized optimization, most existing methods rely solely on information from the current iteration, overlooking the potential of historical information (function values, gradients) to enhance algorithmic performance. While exceptions exist—such as gradient-tracking \cite{nedic17,qu2017harnessing} and momentum-based methods \cite{qu2019accelerated}—they utilize only information from the two most recent iterations.

Leveraging historical information is particularly valuable in decentralized optimization for two reasons. First, historical data offers richer information of the objective function and, in centralized optimization, methods like Anderson acceleration \cite{mai2020anderson} and the proximal bundle method \cite{hiriart1993convex} have demonstrated that strategically leveraging historical information can substantially accelerate convergence and enhance robustness. Second, in many decentralized scenarios, either the communication or the computation is expensive, e.g., computation in decentralized model training and communication in UAV networks. However, information from historical iterations is readily available and requires no additional communication or computation to acquire. To summarize, historical information does not require additional computation or communication to acquire, but may significantly improve the performance. 

Motivated by this, we exploit historical information to accelerate decentralized optimization methods. In particular, we first adopt the proximal bundle framework \cite{hiriart1993convex} to incorporate the objective values and gradients at historical iterates, which gives a bundle surrogate of the function. Then, we integrate the bundle surrogate function with the Proximal DGD (Prox-DGD) method \cite{zeng2018nonconvex}, yielding a Decentralized Proximal Bundle Method (DPBM). Specifically, the update of Prox-DGD involves a surrogate function of the objective function, and we simply replace it with the bundle surrogate function. We also extend DPBM to the asynchronous and stochastic settings to broaden its applicability. We theoretically analysed the convergence of the proposed methods, and it is worth highlighting that both the asynchronous DPBM and its stochastic variant can converge with \emph{fixed step-sizes that do not rely on any delay information}. Numerical experiments on classification problems demonstrate that by using historical information, our methods yield not only faster convergence, but also stronger robustness in step-sizes.

The contributions of this paper are threefold:
\begin{enumerate}[1)]
    \item We investigate the use of historical information to accelerate decentralized optimization and incorporate it into the Prox-DGD method to enhance the performance. While some existing methods also leverage such information, they are typically limited to the two most recent iterations \cite{wang2017decentralized,nedic17,qu2019accelerated}. In contrast, our methods can leverage information from an arbitrary number of past iterations.
    \item We theoretically prove the convergence of the proposed methods under mild conditions.
    \item Our convergence results in the asynchronous setting are superior to most existing works. First, both the asynchronous DPBM and its stochastic variant are guaranteed to converge with \emph{fixed step-sizes that do not rely on any delay information} under proper conditions. However, in most existing asynchronous methods \cite{Zhu2024,zhou2025asynchronous,zhang2019fully,tian2020achieving}, the step-sizes rely on an unknown and usually large upper bound of delays, which results in difficult step-size determination and conservative step-sizes, and slows down the algorithm. Second, the asynchronous DPBM and its stochastic variant can address objective functions that are non-quadratic, non-smooth, and non-Lipschitz continuous, while most existing asynchronous methods only allow for smooth or Lipschitz continuous objective functions \cite{zhou2025asynchronous,Zhu2024,zhang2019fully,tian2020achieving,ubl2021totally}.
\end{enumerate}

\subsection*{Paper Organization and Notation}

The remainder of the paper is organized as follows. Section \ref{sec:problem_and_DSD} formulates the problem, and Section \ref{sec:alg} develops DPBM and its asynchronous and stochastic variant. Section \ref{sec:conv_ana} presents the convergence results, Section \ref{sec:numerical_exp} numerically evaluates the performance of the proposed methods, and Section \ref{sec:conclusion} concludes the paper.

For ease of representation, we denote $\N_0$ as the set of natural numbers including $0$, and for each integer $m>0$,
\[[m] \overset{\triangle}{=}  \{1,\ldots, m\}.\]
We use $\mb{0}_d$, $\mb{0}_{d\times d}$, and $I_d$ to denote the $d$-dimensional all-zero vector, the $d\times d$ all-zero matrix, and the $d\times d$ identity matrix, respectively, and ignore their subscripts when it is clear from the context. For any two matrices $A,B$, we use $A\otimes B$ to represent their Kronecker product. For any set $\mc{D}$, we use $|\mc{D}|$ to denote the number of elements in $\mc{D}$. For any function $f:\mathbb{R}^d\rightarrow\mathbb{R}$, we say that it is $L$-smooth for some $L>0$ if it is differentiable and
	\begin{equation*}
	\|\nabla f(y)-\nabla f(x)\|\le L\|y-x\|,~\forall x,y\in\mathbb{R}^d,
	\end{equation*}
	it is $\mu$-strongly convex for some $\mu>0$ if
    \begin{equation*}
	f(y)-f(x)-\langle g_x, y-x\rangle\!\ge \frac{\mu}{2}\|y-x\|^2,~\forall x,y\in\R^d, g_x\!\in\!\partial f(x),
	\end{equation*}
    and it is coercive if $\lim_{\|x\|\rightarrow +\infty} f(x) = +\infty$. For any $f:\R^d\rightarrow \R\cup\{+\infty\}$ and $x\in\R^d$, we define
    \[\prox_{f}(x) \overset{\triangle}{=}  \underset{z\in\R^d}{\operatorname{\arg\;\min}}~f(z)+\frac{1}{2}\|z-x\|^2.\]

\section{Problem Formulation}\label{sec:problem_and_DSD}
This section discusses the potential of utilizing historical information to enhance decentralized optimization methods. To this end, we first show how historical information can be used to improve optimization methods by reviewing the centralized proximal bundle method \cite{hiriart1993convex}. Then, we introduce decentralized optimization and discuss the potential of historical information in accelerating decentralized methods.

\subsection{Proximal Bundle Method}\label{ssec:PBM}

Exploiting historical information to accelerate convergence is typical in optimization methods, with notable examples including Anderson acceleration \cite{Anderson1965,mai2020anderson}, the cutting-plane method \cite{kelley1960cutting}, and the proximal bundle method \cite{hiriart1993convex,cederberg2025}. These methods have demonstrated competitive performance compared to their counterparts that use information only from the current iteration. In comparison to the Anderson acceleration and the cutting-plane method, the proximal bundle method is particularly robust in its parameters.

When minimizing a differentiable function $f:\R^d\rightarrow\R$, the proximal bundle method takes the form of
\begin{equation}\label{eq:bundle_update}
    x^{k+1} = 
    \underset{x\in\R^d}{\operatorname{\arg\;\min}}~f^k(x).
\end{equation}
In \eqref{eq:bundle_update},
\begin{equation}\label{eq:general_PB}
    f^k(x) = m^k(x)+\frac{\|x-x^k\|^2}{2\gamma^k}
\end{equation}
where $m^k$ is a minorant of $f$ and $\gamma^k>0$ is the step-size. A typical example of $m^k$ is the cutting-plane model (see Fig. \ref{fig:cp-PBM})
\begin{equation}\label{eq:minorant_pb}
    m^k(x) = \max\{f(x^t)+\langle \nabla f(x^t), x - x^t\rangle, t\in\mc{I}^k\},
\end{equation}
where $\mc{I}^k\subseteq [k]$ is a subset of historical indexes including $k$. Letting $m^k=\{k\}$ in 
\eqref{eq:minorant_pb} yields the steepest descent method and, if $f$ is convex, then compared to the setting $\mc{I}^k=\{k\}$, model \eqref{eq:minorant_pb} with a larger size of $\mc{I}^k$ approximates $f$ better: since $f(x)\ge m^k(x)\ge f(x^k)+\langle\nabla f(x^k), x-x^k\rangle$, we have
\begin{equation}\label{eq:fit_better}
    f(x)-m^k(x)\le f(x)-(f(x^k)+\langle\nabla f(x^k), x-x^k\rangle),
\end{equation}
which is also clear from Fig. \ref{fig:cp-PBM}. The higher approximation accuracy not only leads to faster convergence, but also yields stronger robustness in the step-size $\gamma^k$.

\begin{figure}
    \centering
    \includegraphics[width=0.5\linewidth]{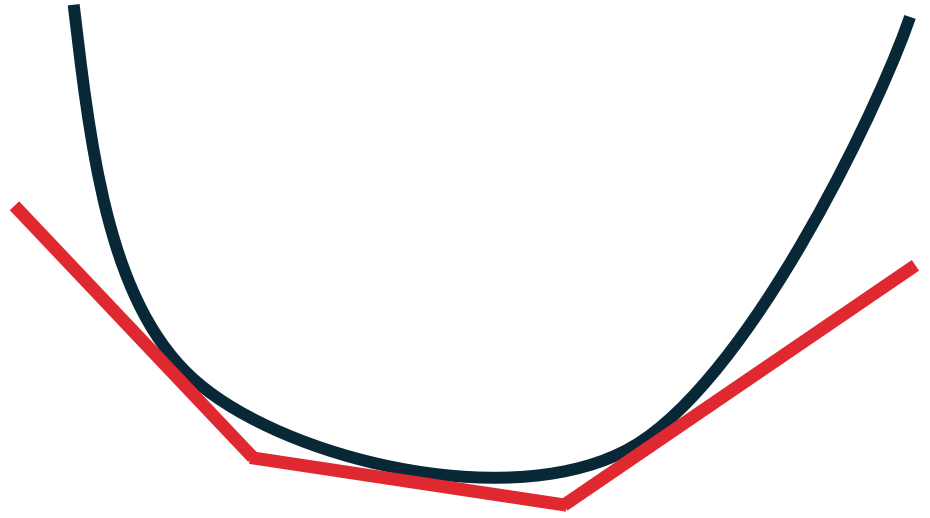}
    \caption{The cutting-plane model (the black curve represents $f$ and the red curve is the model).}
    \label{fig:cp-PBM}
\end{figure}

\subsection{Decentralized Optimization}

Consider a network $\mc{G}=(\mc{V}, \mc{E})$ of $n$ nodes, where $\mc{V}=[n]$ is the vertex set and $\mc{E}\subseteq \mc{V}\times \mc{V}$ is the edge set. Each node $i$ possesses two local objective functions, including a differentiable $f_i:\R^d\rightarrow\R$ and a possibly non-differentiable $h_i:\R^d\rightarrow\R\cup\{+\infty\}$, and all the nodes collaborate to solve
\begin{equation}\label{eq:prob}
\begin{split}
        \underset{x_i\in\R^d,i\in \mc{V}}{\operatorname{minimize}} ~~&~ \sum_{i\in \mc{V}} \phi_i(x_i)\overset{\triangle}{=} (f_i(x_i)+h_i(x_i))\\ \operatorname{subject~to}~&~x_1 = \ldots = x_n,
\end{split}
\end{equation}
which has wide applications, e.g., decentralized learning \cite{lian2017can} and cooperative control \cite{Liu2023}. Typical settings of $h_i$ include the $\ell_1$-norm and the indicator function of a convex set.

We consider problem \eqref{eq:prob} under the following standard assumptions \cite{wu2017decentralized,Wu25}.
\begin{assumption}\label{asm:convex}
    For each $i\in\mc{V}$, the functions $f_i$ and $h_i$ are proper, closed, and convex, and $f_i$ is differentiable.
\end{assumption}
\begin{assumption}\label{asm:graph}
    The network $\mc{G}$ is undirected and connected.
\end{assumption}

We argue that exploiting historical information is particularly valuable in enhancing decentralized optimization methods due to two reasons. First, as discussed in Section \ref{ssec:PBM}, proper use of historical information (e.g., the proximal bundle method \cite{hiriart1993convex}) can substantially accelerate convergence and enhance robustness of optimization algorithms. Second, in many decentralized scenarios, either the computation of certain information (e.g., gradient in model training) or the communication (e.g., in UAV networks) is expensive, while information from historical iterations is readily available and requires no additional computation or communication to acquire. Despite these, most existing decentralized optimization methods use only information from the current \cite{shi2015extra,yuan2016convergence,assran2020asynchronous} or the two most recent iterations \cite{wang2017decentralized,nedic17,qu2019accelerated}.

\section{Algorithm Development}\label{sec:alg}

This section first designs a synchronous decentralized proximal bundle method and then extends it to the asynchronous and stochastic settings.

\subsection{Decentralized Proximal Bundle Method}

We adapt the proximal bundle method to the proximal decentralized gradient descent method (Prox-DGD), which gives the Decentralized Proximal Bundle Method (DPBM).

\subsubsection{Prox-DGD} It is a classical algorithm for solving problem \eqref{eq:prob}. To introduce Prox-DGD, we define $\bx=(x_1^T,\ldots,x_n^T)^T$, $f(\bx)=\sum_{i=1}^n f_i(x_i)$, $h(\bx)=\sum_{i=1}^n h_i(x_i)$, and $\mb{W}=W\otimes I_d$ where $W\in\R^{n\times n}$ is an averaging matrix associated with $\mc{G}$.
\begin{definition}
    We say $W\in\mathbb{R}^{n\times n}$ is an averaging matrix associated with $\mc{G}$ if it is symmetric, doubly stochastic, and $w_{ij}>0$ for $j\in \mc{N}_i\cup\{i\}$ and $w_{ij}=0$ otherwise, where $\mc{N}_i=\{j|~\{i,j\}\in\mc{E}\}$ is the neighbor set of node $i$.
\end{definition}

Prox-DGD for solving problem \eqref{eq:prob} updates as
\begin{equation}\label{eq:DPSD}
    \bx^{k+1} = \prox_{\alpha \h}(\mb{W}\bx^k - \alpha \nabla \mb{f}(\bx^k)),
\end{equation}
where $\alpha>0$ is the step-size and the prox operator is defined at the end of Section \ref{sec:intro}. Due to the simplicity and effectiveness of Prox-DGD, it is extended to many other settings, such as directed networks \cite{nedic2009distributed}, stochastic optimization \cite{koloskova2020unified}, and the asynchronous setting \cite{Wu25}. In particular, several recent works demonstrate the effectiveness of its stochastic extension in model training \cite{Wang2025,koloskova2020unified}.

\subsubsection{Prox-DGD from a surrogate perspective} A key step of our algorithm development is to view Prox-DGD as minimizing a surrogate function per iteration. In particular, Prox-DGD \eqref{eq:DPSD} can be equivalently expressed as the following form:
\begin{equation}\label{eq:model-based-opt}
    \bx^{k+1} = \underset{\bx\in\R^{nd}}{\operatorname{\arg\;\min}}~\f^k(\bx)+\mb{h}(\bx)+\mb{p}^k(\bx),
\end{equation}
where $\f^k(\bx)=\f(\bx^k)+\langle\nabla \f(\bx^k), \bx-\bx^k\rangle+\frac{1}{2\alpha}\|\bx-\bx^k\|^2$ is a \emph{model} (namely, surrogate function) of $\f$ and $\mb{p}^k(\bx) = \mb{p}(\bx^k)+\langle \nabla \mb{p}(\bx^k), \bx-\bx^k\rangle$ is a \emph{model} of $\mb{p}(\bx)=\frac{1}{2\alpha}\bx^T(I-\mb{W})\bx$. In general, Prox-DGD cannot converge to the optimum of problem \eqref{eq:prob} but to the optimum of the following problem under proper conditions \cite{Wu25}:
\begin{equation}\label{eq:penal_prob}
\underset{\bx\in\R^{nd}}{\operatorname{minimize}}~~\mb{f}(\bx) + \h(\bx)+\mb{p}(\bx).
\end{equation}
Although problems \eqref{eq:prob} and \eqref{eq:penal_prob} are not equivalent, their optimality gap can be arbitrarily small \cite{Wu25} by choosing a small $\alpha$. Moreover, for many applications such as model training, only a moderate accuracy is required. For the update \eqref{eq:model-based-opt}, more accurate $\f^k(\bx)$ and $\mb{p}^k(\bx)$ typically lead to faster convergence.

\subsubsection{Decentralized proximal bundle method} The main philosophy of our algorithm development is to adopt a more accurate model $\f^k$ in \eqref{eq:model-based-opt}. For distributed implementation, we require $\f^k(\bx)$ to be separable:
\begin{equation}\label{eq:model_f_history}
    \f^k(\bx) = \sum_{i=1}^n f_i^k(x_i),
\end{equation}
for some $f_i^k:\R^d\rightarrow \R$. Then, \eqref{eq:model-based-opt} can be expressed as
\begin{equation}\label{eq:model_individual}
        x_i^{k+1} = \underset{x_i\in\R^d}{\operatorname{\arg\;\min}}~f_i^k(x_i)+h_i(x_i)+\langle \nabla_i \mb{p}(\bx^k), x_i\rangle,
\end{equation}
where $\nabla_i \mb{p}(\bx^k)=\frac{1}{\alpha}\sum_{j\in\mc{N}_i} w_{ij}(x_i^k-x_j^k)$ is the $i$th block of $\nabla \mb{p}(\bx^k)$ and the constant terms $\mb{p}(\bx^k)$ and $\langle \nabla_i \mb{p}(\bx^k), x_i^k\rangle$ in \eqref{eq:model-based-opt} is dropped since it does not affect the solution. It is clear that the update \eqref{eq:model_individual} can be implemented in a distributed way. For $\mb{p}^k$, since $\mb{p}(\bx)$ is non-separable, it is difficult to find a model that is more effective than the one in Prox-DGD and, at the same time, enables the distributed implementation of \eqref{eq:model-based-opt}.

Next, we adapt the proximal bundle model \eqref{eq:general_PB} introduced for centralized optimization to the decentralized setting to design the surrogate function $f_i^k$ in \eqref{eq:model_individual}. Specifically, we let
\begin{equation}\label{eq:general_PB_ind}
    f_i^k(x_i) = m_i^k(x_i)+\frac{\|x_i-x_i^k\|^2}{2\gamma_i^k},
\end{equation}
where $m_i^k(x_i)$ is a minorant of $f_i$ and $\gamma_i^k>0$. Substituting \eqref{eq:general_PB_ind} into \eqref{eq:model_individual} yields
\begin{equation}\label{eq:syn_PB_alg}
\begin{split}
x_i^{k+1} =&~\underset{x_i\in\R^d}{\operatorname{\arg\;\min}}~m_i^k(x_i)+h_i(x_i)+\frac{\|x_i-x_i^k\|^2}{2\gamma_i^k}\\
&+\frac{1}{\alpha}\langle\sum_{j\in\mc{N}_i} w_{ij} (x_i^k-x_j^k), x_i\rangle.
\end{split}
\end{equation}
Candidates of $m_i^k$ will be discussed in Section \ref{ssec:minorant}. We refer to the algorithm as the Decentralized Proximal Bundle Method (DBPM), whose decentralized implementation is straightforward and is therefore ommited.

\subsection{Asynchronous DPBM}

When implementing the update \eqref{eq:syn_PB_alg} in an asynchronous way, information from neighbouring nodes is, in general, delayed, and not all nodes update at each iteration. To describe the asynchronous update, we use $\mc{K}_i\subseteq \mbb{N}_0$ to denote the set of iterations where node $i$ updates, i.e.,
\begin{equation}
    \mc{K}_i = \{k|~\text{node }i \text{ updates at the }k{\text{th}} \text{ iteration}\}.
\end{equation}
Then, the asynchronous implementation of \eqref{eq:syn_PB_alg} can be described as: For all $k\in\mc{K}_i$, 
\begin{equation}\label{eq:asyn_PB_alg}
\begin{split}
x_i^{k+1} =&~\underset{x_i\in\R^d}{\operatorname{\arg\;\min}}~m_i^k(x_i)+h_i(x_i)+\frac{\|x_i-x_i^k\|^2}{2\gamma_i^k}\\
&+\frac{1}{\alpha}\langle \sum_{j\in\mc{N}_i} w_{ij}(x_i^k-x_j^{s_{ij}^k}), x_i\rangle,
\end{split}
\end{equation}
where $x_j^{s_{ij}^k}$ is the most recent $x_j$ owned by node $i$ with $s_{ij}^k\in [k]$ being its iteration index, $s_{ii}^k=k$, and $k-s_{ij}^k$ represents the delay. If $k\notin \mc{K}_i$, then $x_i^{k+1}=x_i^k$.

Algorithm \ref{alg:ADPBM} details the implementation of the asynchronous DPBM, where each node is activated at discrete time points to update. We ignore the iteration index in the updates for clarity and, for each pair of neighboring nodes $i$ and $j$, we introduce $x_{ij}$ to denote the most recent $x_j$ that node $i$ has received from $j$. Each node $i$ uses a buffer ${\mathcal B}_i$ to store $x_j$'s it received from neighboring nodes, which can be executed regardless of whether $i$ is active or not. When node $i$ is activated, it reads all $x_{ij}$, $j\in\mc{N}_i$ from the buffer, performs the local update
\begin{equation}\label{eq:ADPBM}
\begin{split}
x_i\leftarrow &~\underset{z_i\in\R^d}{\operatorname{\arg\;\min}}~m_i(z_i)+h_i(z_i)+\frac{\|z_i-x_i\|^2}{2\gamma_i}\\
&+\frac{1}{\alpha}\langle\sum_{j\in\mc{N}_i} w_{ij} (x_i-x_{ij}), z_i\rangle
\end{split}
\end{equation}
and broadcasts the new $x_i$ to all its neighbors, and then enters a sleep mode.

\begin{figure*}[!htbp]
    \centering
    \begin{subfigure}[t]{0.3\textwidth}
        \includegraphics[width=0.8\textwidth, valign=t]{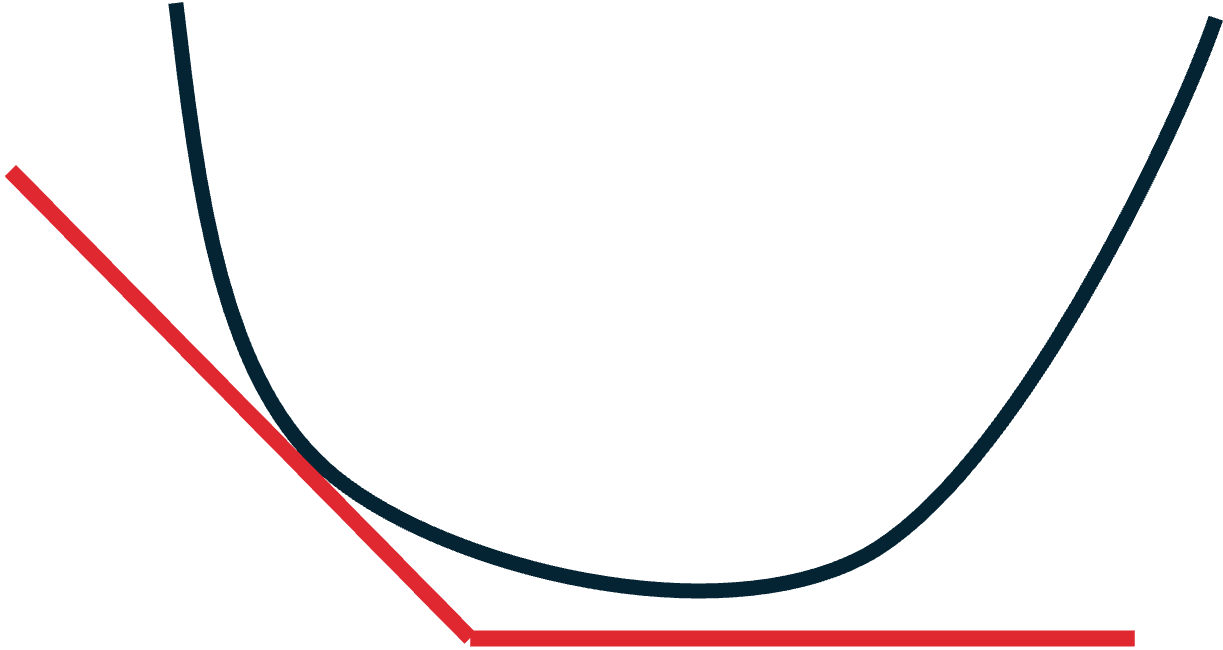}
        \vspace{0.74cm}
        \caption{Polyak}
        \label{fig:sub1}
    \end{subfigure}
    \hfill
    \begin{subfigure}[t]{0.3\textwidth}
        \includegraphics[width=0.8\textwidth, valign=t]{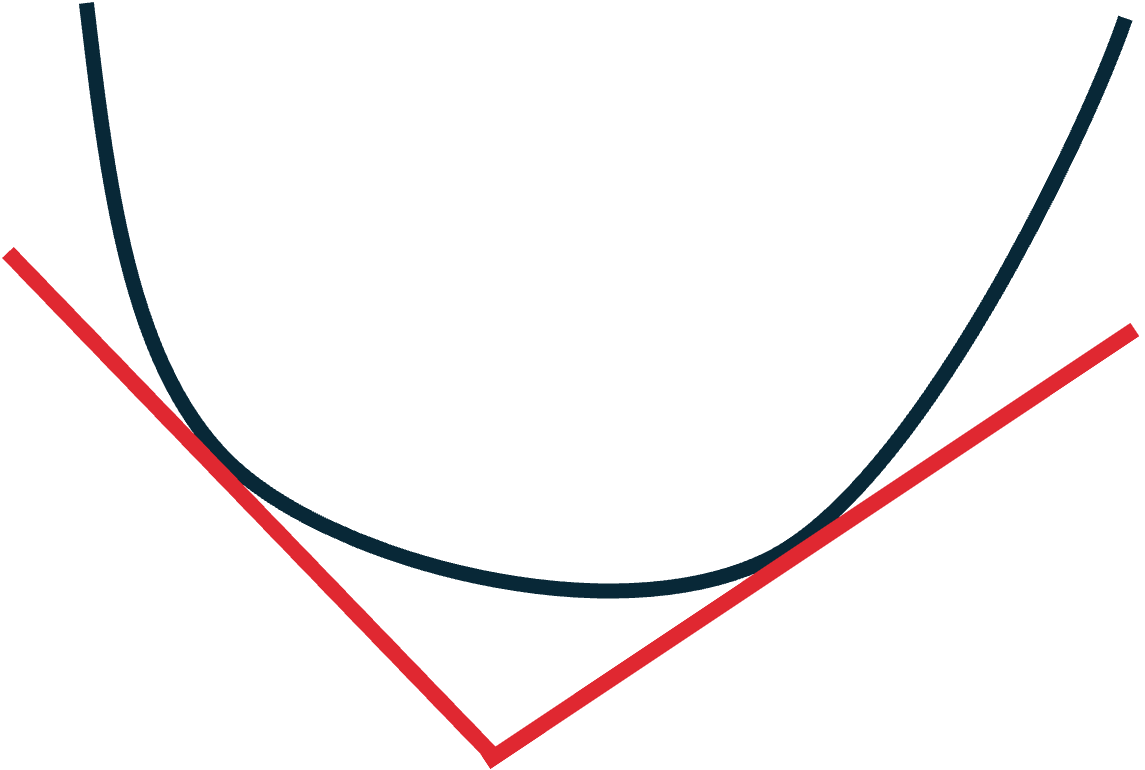}
        \vspace{0.1cm}
        \caption{Cutting-plane}
        \label{fig:sub2}
    \end{subfigure}
    \hfill
    \begin{subfigure}[t]{0.3\textwidth}
        \includegraphics[width=0.8\textwidth, valign=t]{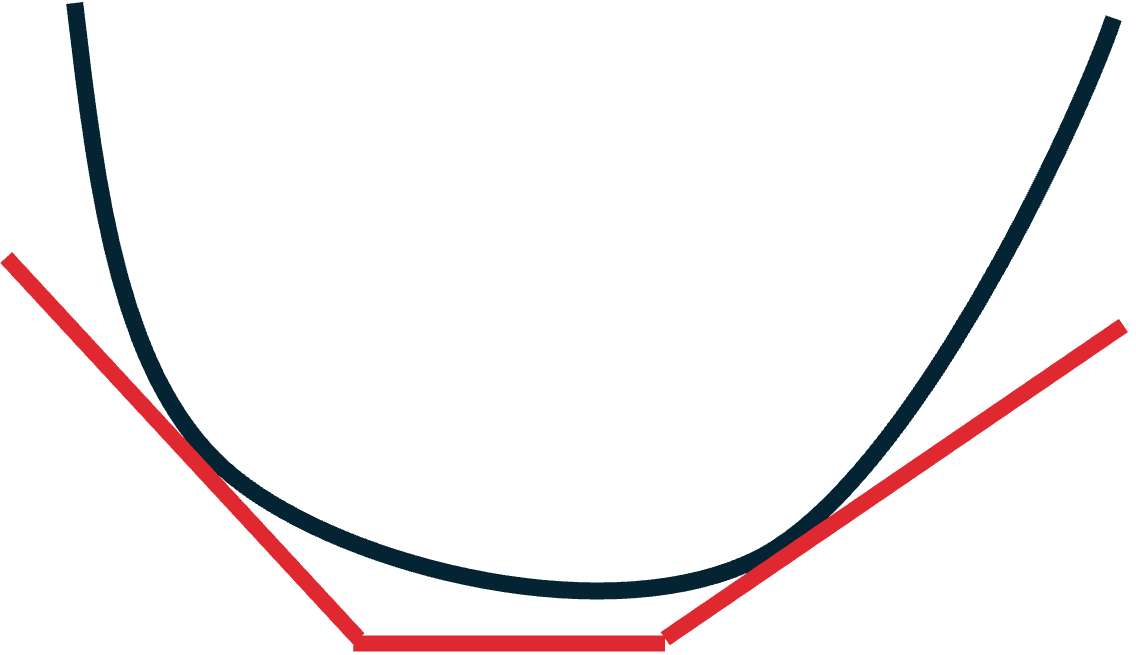}
        \vspace{0.4cm}
        \caption{Polyak cutting-plane}
        \label{fig:sub3}
    \end{subfigure}
    \caption{Surrogate functions in the deterministic Polyak model
    \eqref{eq:pol}, cutting-plane model \eqref{eq:cp}, and the Polyak cutting-plane model \eqref{eq:pol_cp}. The black curve represents $f$ and the red curves are the models.}
\label{fig:three_subfigures}
\end{figure*}

\subsection{Asynchronous Stochastic DPBM}\label{ssec:stochastic_alg}

In data-driven applications, each $f_i$ in problem \eqref{eq:prob} is often defined by a set $\mc{D}_i$ of data samples:
\begin{equation}\label{eq:sto_prob}
	f_i(x_i) = \frac{1}{|\mc{D}_i|}\sum_{\xi\in \mc{D}_i} F_i(x_i;\xi),
\end{equation}
where each $\xi$ is a data
sample, \(F(\cdot;\xi):\mathbb{R}^d\to \mathbb{R}\) is a convex loss function defined by $\xi$, and $f_i(x_i)$ is the loss across the dataset $\mc{D}_i$. When $\mc{D}_i$ is large, computing the exact value and gradient of $f_i$ is often prohibitively expensive and a more reasonable and typical way is to use the stochastic function value and gradient, i.e.,
\begin{equation}\label{eq:sto_est}
    F_i(x_i,\mc{D}_i'),\quad \nabla F_i(x_i,\mc{D}_i')
\end{equation}
for a subset $\mc{D}_i'\subseteq \mc{D}_i$, where \[F_i(x_i,\mc{D}_i')\overset{\triangle}{=} \frac{1}{|\mc{D}_i'|}\sum_{\xi\in \mc{D}_i'} F_i(x_i;\xi).\]
The stochastic function value and gradient \eqref{eq:sto_est} are unbiased estimations of $f_i(x_i)$ and $\nabla f_i(x_i)$ if $\mc{D}_i'$ is drawn from $\mc{D}_i$ uniformly randomly.

When $f_i$ in problem \eqref{eq:prob} takes the form of \eqref{eq:sto_prob}, we construct $m_i^k$ in \eqref{eq:asyn_PB_alg} with stochastic function values and gradients (see Section \ref{ssec:minorant}), which gives the asynchronous stochastic DPBM. The implementation is similar to Algorithm \ref{alg:ADPBM}. 

 \begin{algorithm}[t!]
    \makeatletter
    \renewcommand\footnoterule{%
        \kern-3\p@
    \hrule\@width.4\columnwidth
    \kern2.6\p@}
    \makeatother
    \caption{Asynchronous DPBM}
    \begin{minipage}{\linewidth}
    \renewcommand{\thempfootnote}{\arabic{mpfootnote}}
    
    \label{alg:ADPBM}
		\begin{algorithmic}[1]
			\STATE {\bfseries Initialization:} All the nodes agree on $\alpha>0$, and cooperatively set $w_{ij}$ $\forall \{i,j\}\in\mc{E}$.
			\STATE Each node $i\in\mc{V}$ chooses $x_i\in\mathbb{R}^d$, creates a local buffer $\mc{B}_i$, and shares $x_i$ with all neighbors in $\mc{N}_i$.
			\FOR{each node $i\in \mc{V}$}
			\STATE 
			keep \emph{receiving $x_j$ from neighbors and store $x_j$ in $\mc{B}_i$ until activation}\footnote{In the first iteration, each node $i\in\mc{V}$ can be activated only after it received $x_j$ from all $j\in\mc{N}_i$. If for some $j\in\mc{N}_i$, node $i$ receives multiple $x_j$'s, then it only stores the most recently received one and drop the remaining ones.}.
			\STATE set $x_{ij}=x_j$ for all $x_j\in\mc{B}_i$.
			\STATE empty $\mc{B}_i$.
			\STATE update $x_i$ according to \eqref{eq:ADPBM}.
			\STATE send $x_i$ to all neighbors $j\in\mc{N}_i$.
			\ENDFOR
			\STATE \textbf{Until} a termination criterion is met.
		\end{algorithmic}
\end{minipage}
\end{algorithm} 

\subsection{Candidates of the Minorant}\label{ssec:minorant}

This subsection provides several candidates for the minorant $m_i^k$ in both the deterministic and stochastic settings.

\subsubsection{Deterministic minorants} We impose the following assumption on $m_i^k$, which is typical for proximal bundle methods and is also required for our convergence analysis in Section \ref{sec:conv_ana}.
\begin{assumption}\label{asm:tilde_f}
    For each $i\in\mc{V}$, it holds that
    \begin{enumerate}[i)]
        \item $m_i^k(x_i)$ is convex;
        \item for some $\beta_i\ge 0$ and any $x_i\in\R^d$,
        \begin{equation}\label{eq:mk_key_prop}
            m_i^k(x_i) \le f_i(x_i)\le m_i^k(x_i)+\frac{\beta_i}{2}\|x_i-x_i^k\|^2.
        \end{equation}
    \end{enumerate}
\end{assumption}
Assumption \ref{asm:tilde_f} can be satisfied by the following models.

\begin{itemize}[leftmargin=8pt]
    \item[-] {\bf Polyak model}: The model originates from the Polyak step-size \cite{polyak1987introduction} for gradient descent, and takes the form of
    \begin{equation}\label{eq:pol}
        m_i^k(x_i)=\max\{f_i(x_i^k)+\langle \nabla f_i(x_i^k), x_i-x_i^k \rangle, c_i^f\},
    \end{equation}
    where $c_i^f \le \inf_{x_i} f_i(x_i)$ is a lower bound of $f_i$. In centralized optimization, the steepest descent method with Polyak step-size minimizes a function $f$ by setting $x^{k+1}$ as a minimizer of $m^k(x)$ for all $k\ge 0$, where $m^k(x)$ is given by \eqref{eq:pol} with $f_i$ replaced by $f$. This model and its variants are shown to be particularly effective in stochastic optimization \cite{loizou2021stochastic,wang2023generalized}.
    \vspace{0.2cm}
    \item[-] {\bf Cutting-plane model}: This model takes the maximum of several cutting planes at historical iterates:
    \begin{align}\label{eq:cp}
     \!m_i^k(x_i) \!=\!\max\{f_i(x_i^t)\!+\!\langle \nabla f_i(x_i^t), x_i-x_i^t\rangle,t\in \mc{S}_i^k\},
    \end{align}
    where $\mc{S}_i^k$ is a subset of $\mc{K}_i$ satisfying $k\in \mc{S}_i^k$. The model is adopted in the cutting-plane method \cite{kelley1960cutting} and is also typical in the bundle method \cite{cederberg2025}.
    \vspace{0.2cm}
    \item[-] {\bf Polyak cutting-plane model}: It is natural to combine the Polyak model and the cutting-plane model, which yields the following Polyak cutting-plane model:
    \begin{align}\label{eq:pol_cp}
     m_i^k(x_i) =\!\max\{f_i(x_i^t)\!+\!\!\langle \nabla f_i(x_i^t), x_i-x_i^t\rangle,t\in \mc{S}_i^k,c_i^f\},
    \end{align}
    where all the parameters are introduced below \eqref{eq:pol} or \eqref{eq:cp}.
    \vspace{0.1cm}
    \item[-] {\bf Two-cut model}: The model is defined in an iterative way. Set $m_i^0(x_i)=f_i(x_i^0)+\langle \nabla f_i(x_i^0), x_i-x_i^0\rangle$. For each $k\ge 0$, if $k\notin \mc{K}_i$, $m_i^{k+1}(x_i)=m_i^k(x_i)$; Otherwise,
    \begin{equation}\label{eq:two_cut}
    \begin{split}
        m_i^{k+1}(x_i) =
            &\max\{m_i^{k}(x_i^{k+1})+\langle \hat{g}_i^{k},x_i-x_i^{k+1}\rangle,\\
            & f_i(x_i^{k+1})+\langle \nabla f_i(x_i^{k+1}), x_i-x_i^{k+1}\rangle\},
    \end{split}
    \end{equation}
    where $\hat{g}_i^{k}\in \partial m_i^{k}(x_i^{k+1})$.
\end{itemize}

The above models are originally proposed in centralized optimization, and we adapt them to decentralized optimization. In general, if $m_i^k$ approximates $f_i$ with a higher accuracy, then we can choose a larger $\gamma_i^k$ so that $f_i^k$ in \eqref{eq:general_PB_ind} approximates $f_i$ better and the update makes a larger progress. Fig. \ref{fig:three_subfigures} plots the Polyak model, the cutting-plane model, and the Polyak cutting-plane model, where the Polyak cutting-plane model incorporates more information and yields a higher accuracy. 

Lemma \ref{lemma:det_key_prop} shows that Assumption \ref{asm:tilde_f} can be satisfied by the aforementioned models under proper assumptions on $f_i$.
\begin{lemma}\label{lemma:det_key_prop}
    Suppose that Assumption \ref{asm:convex} holds and each $f_i$ is $\beta_i$-smooth. For the models \eqref{eq:pol}--\eqref{eq:two_cut}, Assumption \ref{asm:tilde_f} holds.
\end{lemma}
\begin{proof}
See Appendix \ref{proof:lemma_det_key_prop}.
\end{proof}

\subsubsection{Stochastic minorants} For problem \eqref{eq:prob} with data-driven $f_i$ (see \eqref{eq:sto_prob}), we construct the minorant $m_i^k$ with stochastic function values and gradients, as discussed in Section \ref{ssec:stochastic_alg}.

Similar to the deterministic case, we make assumptions on the stochastic minorant $m_i^k$. To introduce the assumption, we suppose that for all $i\in\mc{V}$ and $k\in\mc{K}_i$, the stochastic functions and gradients are computed with respect to a random batch $\mc{D}_i^k\subseteq \mc{D}_i$ of samples, and define
\begin{equation}\label{eq:Fk_def}
    \mc{F}^k = \cup_{i\in\mc{V}}~\{\mc{D}_i^t|~t\le k, t\in\mc{K}_i\}.
\end{equation}
We also suppose that the optimal solution to problem \eqref{eq:penal_prob} exists, which holds under mild conditions (e.g., coercive $\phi_i$) \cite{Wu25}. Moreover, the optimal solution of problem \eqref{eq:penal_prob} is typically sub-optimal to problem \eqref{eq:prob} with error bound $O(\sqrt{\alpha})$ or $O(\alpha)$ under certain conditions \cite[Lemmas 2, 3]{Wu25}.

\begin{assumption}\label{asm:opt_exist}
    The optimal solution of problem \eqref{eq:penal_prob} exists.
\end{assumption}

Now, we are ready to introduce our assumption on $m_i^k$.
\begin{assumption}\label{asm:stoch_tilde_f}
    For each $i\in\mc{V}$ and $k\in\mc{K}_i$, it holds that
    \begin{enumerate}[label= \roman*), leftmargin = 0.7cm]
        \item $m_i^k(x_i)$ is convex;
        \item there exists $\beta_i,\epsilon_i>0$ such that for any $x_i\in\R^d$, 
        \begin{equation}\label{eq:sto_smooth_tilde_sigma}
            f_i(x_i)\le \mbb{E}[m_i^k(x_i)+\frac{\beta_i}{2}\|x_i-x_i^k\|^2|~\mc{F}^k]+\epsilon_i;
        \end{equation}
        \item for an optimal solution $\bx^\star=((x_1^\star)^T,\ldots,(x_n^\star)^T)^T$ of problem \eqref{eq:penal_prob} and some $\epsilon_i^\star\ge 0$,
        \begin{equation}\label{eq:sto_smooth_tilde_sigma_star}
            \mbb{E}[m_i^k(x_i^\star)|~\mc{F}^k]-f_i(x_i^\star) \le \epsilon_i^\star.
        \end{equation}
    \end{enumerate}
\end{assumption}
In the stochastic setting, Assumption \ref{asm:tilde_f} is difficult to guarantee even in the expectation sense. Instead, the conditions in Assumption \ref{asm:stoch_tilde_f} are more relaxed and, under the following assumption, can be satisfied by stochastic variants of \eqref{eq:pol}--\eqref{eq:two_cut}.



\begin{assumption}\label{asm:stochasticity}
For each $i\in\mc{V}$, it holds that
\begin{enumerate}[i)]
    \item for any $\xi\in\mc{D}_i$, $F_i(x_i;\xi)$ is proper, closed, and convex;
    \item each $\mc{D}_i^k$, $k\in\mc{K}_i$ is uniformly randomly drawn from $\mc{D}_i$;
    \item there exist $\sigma,\sigma^\star\ge 0$ such that for each $i\in\mc{V}$ and $k\in\mc{K}_i$,
    \begin{align}
        &\mbb{E}[\|\nabla F_i(x_i^k;\mc{D}_i^k)-\nabla f_i(x_i^k)\|^2|~ \mc{D}_i^k]\le \sigma^2,\label{eq:sigma}\\
        &\mbb{E}[(F_i(x_i^\star; \mc{D}_i^k) - f_i(x_i^\star))^2|~ \mc{D}_i^k] \le (\sigma^\star)^2.\label{eq:sigma_f}
    \end{align}
\end{enumerate}
\end{assumption}

\begin{lemma}\label{lemma:assum_sto}
    Suppose that Assumptions \ref{asm:convex}, \ref{asm:opt_exist}, \ref{asm:stochasticity} hold and $F_i(x_i;\xi)$ is $L_i$-smooth for all $i\in\mc{V}$ and $\xi\in\mc{D}_i$. For the models \eqref{eq:pol}--\eqref{eq:two_cut} where $f_i(x_i^t)$ and $\nabla f_i(x_i^t)$ are replaced by $F_i(x_i^t;\mc{D}_i^t)$ and $\nabla F_i(x_i^t;\mc{D}_i^t)$, respectively, and $c_i^f\le\min_{\xi\in \mc{D}_i} F_i(x_i;\xi)$, Assumption \ref{asm:stoch_tilde_f} holds. Specifically, 
    \[\epsilon_i = \frac{\sigma^2}{2L_i}, \quad \beta_i=2L_i\] 
    for all four models, $\epsilon_i^\star=\sigma^\star$ for \eqref{eq:pol}, $\epsilon_i^\star=\sqrt{\max_{k\in\mc{K}_i}|\mc{D}_i^k|}\sigma^\star$ for \eqref{eq:cp}--\eqref{eq:pol_cp}, and $\epsilon_i^\star=\max_{\xi\in\mc{D}_i} F(x_i^\star,\xi)-f_i(x_i^\star)$ for \eqref{eq:two_cut}.
\end{lemma}
\begin{proof}
See Appendix \ref{proof:assum_sto}.
\end{proof}

\subsection{Solve the Subproblem \eqref{eq:asyn_PB_alg}}

The efficiency of the algorithm heavily relies on that of solving the subproblem \eqref{eq:asyn_PB_alg}, which reduces to \eqref{eq:syn_PB_alg} in the synchronous setting. In this subsection, we introduce an efficient dual approach to solve \eqref{eq:asyn_PB_alg} for $m_i^k$ in \eqref{eq:pol}--\eqref{eq:two_cut}. 

In \eqref{eq:pol}--\eqref{eq:two_cut}, $m_i^k$ are piece-wise linear functions. For clarity, we focus on the following simplified problem of \eqref{eq:asyn_PB_alg}:
\begin{equation}\label{eq:simplified_prob}
    \underset{x\in\R^d}{\operatorname{minimize}}~\max\{a^tx+b^t, t\in [T]\}+h(x)+\frac{1}{2\gamma}\|x-\tilde{x}\|^2,
\end{equation}
where we omit the subscript $i$ and superscript $k$ in $x_i$, $h_i$, and $\gamma_i^k$ for simplicity, use $a^tx+b^t$ to represent the $t$th affine function in $m_i^k$, and set \begin{equation}\label{eq:tilde_x}
    \tilde{x}=x_i^k-\frac{\gamma_i^k}{\alpha}\sum_{j\in\mc{N}_i} w_{ij}(x_i^k-x_j^{s_{ij}^k}).
\end{equation}
Further, we transform problem \eqref{eq:simplified_prob} into the equivalent problem
\begin{equation}\label{eq:lin_prob}
    \begin{split}
        \underset{x\in\R^d,y\in\R}{\operatorname{minimize}}~~&~ y+h(x)+\frac{\|x-\tilde{x}\|^2}{2\gamma}\\
        \operatorname{subject~to}~&~a^tx+b^t\le y,~t\in [T].
    \end{split}
\end{equation}
For problem \eqref{eq:lin_prob}, the variable dimension of its Lagrange dual problem is $T$, which is typically much smaller than the dimension $d$ of the primal variable in practical implementations. For example, $T=2$ in the Polyak model \eqref{eq:pol} and the two-cut model \eqref{eq:two_cut}, and $T=2,5,10$ in the experiments in \cite{cederberg2025} which adopts the cutting-plane model \eqref{eq:cp}. This makes dual approaches particularly attractive for solving \eqref{eq:lin_prob} since problems with lower dimension are often easier to solve. 

To introduce the dual problem, we define $\mb{a}\in\R^T$ and $\mb{b}\in \R^{d\times T}$ as the stackings of $\{a^{t}\}_{t=1}^T$ and $\{b^{t}\}_{t=1}^T$, respectively. 
Also define the Moreau envelope of $h$ as: for any $z\in\R^d$,
\begin{align*}
    M_{\gamma h}(z) &= \min_{x\in\R^d}~h(x)+\frac{1}{2\gamma}\|x-z\|^2.
\end{align*}
Then, by \cite{cederberg2025}, the dual problem of \eqref{eq:lin_prob} is
\begin{equation}\label{eq:dual_sub}
    \begin{split}
        \operatorname{maximize}~~&~q(v)\\
        \operatorname{subject~to}~&~\mb{1}^Tv = 0,~v\ge \mb{0},
    \end{split}
\end{equation}
where $v\in\R^T$ is the dual variable and
\[\begin{split}
    q(v)=&\frac{\|\tilde{x}-\gamma \mb{b}^Tv\|^2}{2\gamma}+M_{\gamma h}(z-\gamma \mb{b}^Tv)+\mb{a}^Tv
\end{split}\]
with $\tilde{x}$ defined in \eqref{eq:tilde_x}. Note that the Moreau envelope $M_{\gamma h}$ is smooth \cite[Proposition 12.29]{bauschke2011convex} and its gradient is
\[\begin{split}
    &\nabla M_{\gamma h}(\tilde{x}-\gamma \mb{b}^Tv)=-\mb{b}(\tilde{x}-\gamma \mb{b}^Tv-\operatorname{prox}_{\gamma h}(\tilde{x}-\gamma \mb{b}^Tv)).
\end{split}\]
Therefore, $q$ is also smooth and its gradient is easy to solve when the proximal operation on $h$ is easy to compute. Moreover, the constraint set of \eqref{eq:dual_sub} is the simplex, and projection onto it can be implemented efficiently \cite{condat2016fast}. As a result, we can use projected gradient descent with adaptive step-sizes \cite{malitsky2024adaptive} or FISTA \cite{beck2009fast} to solve problem \eqref{eq:dual_sub}, which requires only the computation of $\nabla q$ and the projection onto the feasible region of problem \eqref{eq:dual_sub}. Specifically, when $h = \|\cdot\|_1$, $d = 100, 000$, and $T = 15$, applying FISTA
to solve \eqref{eq:dual_sub} only takes $19 \sim 40$ iterations and $0.1 \sim 0.2$ second to reach an accuracy of $10^{-7}$ in a PC with the Apple M3 8-core CPU. Once the optimum $v^{\text{opt}}$ of \eqref{eq:dual_sub} is solved, according to \cite{cederberg2025}, the optimum of \eqref{eq:simplified_prob} can be recovered by
\[x^{\text{opt}} = \operatorname{prox}_{\gamma h}(\tilde{x}-\gamma \mb{b}^Tv^{\text{opt}}).\]

\section{Convergence Analysis}\label{sec:conv_ana}

This section analyses the convergence of the proposed algorithms. Since the asynchronous setting is more general than the synchronous one, we only analyse the convergence of the asynchronous DPBM and the asynchronous stochastic DPBM, which encompasses the convergence analysis of DPBM and the stochastic DPBM.

We consider two standard asynchrony models, including the partial asynchrony model and the total asynchrony model \cite{bertsekas2015parallel}.


    \begin{assumption}[partial asynchrony]\label{asm:partialasynchrony}
		There exist non-negative integers $B$ and $D$ such that
		\begin{enumerate}[i)]
			\item for every $i\in\mc{V}$ and for every $k\in\N_0$, at least one element in the set $\{k,\ldots,k+B\}$ belongs to $\mc{K}_i$;
			\item there holds
			\begin{equation}\label{eq:sijk_range}
			k-D \le s_{ij}^k \le k
			\end{equation}
			for all $i\in\mc{V}$, $j\in\mc{N}_i$, and $k\in \mc{K}_i$.
		\end{enumerate}
	\end{assumption}
    The parameters $B$ and $D$ characterize the minimum update frequency and the maximum delay, respectively, and if $B=D=0$, then Assumption \ref{asm:partialasynchrony} yields the synchronous setting where all nodes update at every iteration and there are no information delays.

	\begin{assumption}[total asynchrony]\label{asm:totalasynchrony}
		The following holds:
		\begin{enumerate}[i)]
			\item $\mc{K}_i$ is an infinite subset of $\N_0$ for each $i\in\mc{V}$;
			\item $\lim_{k\rightarrow+\infty} s_{ij}^k= +\infty$ for any $i\in\mc{V}$ and $j\in\mc{N}_i$.
		\end{enumerate}
	\end{assumption}
    Assumption \ref{asm:totalasynchrony} is more relaxed than Assumption \ref{asm:partialasynchrony}, where condition i) requires each node to update infinitely many times, and condition ii) allows the information delays $k-s_{ij}^k$ to grow arbitrarily large, but requires that old information must eventually be purged from the system. 

\subsection{Deterministic Algorithm}

We first study the setting where $m_i^k$ is constructed using deterministic function values and gradients. To this end, let $\bx^\star$ be an arbitrary optimal solution of problem \eqref{eq:penal_prob}.

\begin{theorem}\label{thm:epsilon_0}
    Suppose that Assumptions \ref{asm:convex}-- \ref{asm:opt_exist} hold. Let $\{\bx^k\}$ be generated by the update \eqref{eq:asyn_PB_alg}. If
    \begin{equation}\label{eq:step-cond}
        \gamma_i^k=\gamma_i \in \left[0, \frac{1}{\beta_i+\frac{1-w_{ii}}{\alpha}}\right)
    \end{equation}
    for all $i\in\mc{V}$ and $k\in\N_0$, then
    \begin{enumerate}[1)]
        \item under Assumption \ref{asm:partialasynchrony}, $\{\bx^k\}$ converges to an optimal solution of problem \eqref{eq:penal_prob}. If, in addition, each $\phi_i$ is $\theta_i$-strongly convex for some $\theta_i>0$, then
        \begin{equation}\label{eq:conv_linear}
            \|x_i^{k}-x_i^\star\|^2\le \rho^{\lfloor k/(B+D+1)\rfloor}\max_{j\in\mc{V}}\|x_j^0-x_j^\star\|^2,
        \end{equation}
        where $\rho=1/(1+\min_{i\in\mc{V}} \gamma_i\theta_i)\in(0,1)$.
        \item under Assumption \ref{asm:totalasynchrony}, if each $\phi_i$ is $\theta_i$-strongly convex for some $\theta_i>0$, then $\{\bx^k\}$ converges to an optimal solution of problem \eqref{eq:penal_prob}.
        \end{enumerate}
\end{theorem}
\begin{proof}
See Appendix \ref{proof:epsilon_0}.
\end{proof}

\begin{remark}
    In asynchronous decentralized optimization, most works use either diminishing step-sizes \cite{zhang2019asyspa,doan2017convergence,assran2020asynchronous,zhang2019fully,spiridonoff2020robust} or fixed step-sizes \cite{wu2017decentralized,tian2020achieving, kungurtsev2023decentralized} that rely on and decrease with an upper bound $\tau$ of delays. The diminishing step-sizes used in \cite{zhang2019asyspa,spiridonoff2020robust,doan2017convergence,kungurtsev2023decentralized,assran2020asynchronous} decrease rapidly and can often lead to slow practical convergence, especially for deterministic optimization problems. The delay-dependent, fixed step-sizes in \cite{wu2017decentralized, tian2020achieving,zhang2019fully} are difficult to determine because the delay bound is typically unknown in practice, and they can be overly small when the delay bound $\tau$ is large. For example, the maximum observed delay can be as large as $160$ in a $30$-node network \cite{Wu25}, while the step-sizes in \cite{zhang2019fully, wu2017decentralized,tian2020achieving} are of the orders $O((n^2\tau)^{-n\tau})$, $O(\frac{1}{\tau+1})$, and $O(\eta^{n\tau})$ where $\eta=\min_{i\in\mc{V}} w_{ii}\in(0,1)$, respectively. In contrast, our delay-independent, non-diminishing step-size \eqref{eq:step-cond} is easy to determine since it does not include any delay information, and it is less conservative since it does not decrease with the iteration index and does not become overly small when the information delays are large. The less conservative step-sizes often yield faster practical convergence \cite{Wu25}.
\end{remark}

\begin{remark}
    Our algorithms can handle problem \eqref{eq:prob} with non-quadratic and non-smooth objective functions, while most existing works on asynchronous decentralized optimization \cite{zhou2025asynchronous,Zhu2024,zhang2019fully,tian2020achieving,ubl2021totally} cannot. Specifically, \cite{zhou2025asynchronous,Zhu2024,zhang2019fully,tian2020achieving} require $h_i\equiv 0$ and \cite{ubl2021totally} considers quadratic objective functions. Although the method in \cite{wu2017decentralized} can simultaneously address non-quadratic and non-smooth objective functions, it uses a probabilistic model to characterize the node that updates at each iteration, which is easier to analyse but less practical than the partial and total asynchrony models.
\end{remark}

Next, we consider the case where $\beta_i$ is unknown, and use the parameter $\beta_i^k>0$ that satisfies
\begin{align}
    f_i(x_i^{k+1}) &= m_i^k(x_i^{k+1})+\frac{\beta_i^k\|x_i^{k+1}-x_i^k\|^2}{2}.\label{eq:beta_ik}
\end{align}
Clearly,
\begin{equation}\label{eq:bounded_betai}
    \beta_i^k\le \beta_i.
\end{equation}

\begin{theorem}\label{thm:epsilon_0_adapt}
    Suppose that Assumptions \ref{asm:convex}--\ref{asm:opt_exist} hold. If for all $i\in\mc{V}$, there exists $\eta_i\in (0,1)$ such that
    \begin{align}
        &\gamma_i^k \le \frac{\eta_i}{\beta_i^k+\frac{1-w_{ii}}{\alpha}},\label{eq:ada_step-cond}\\
        &\min_{k\in\mc{K}_i} \gamma_i^k>0,\label{eq:ada_step-cond_lowerbound}
    \end{align}
    then all the convergence results in Theorem \ref{thm:epsilon_0} hold, with $\rho$ in \eqref{eq:conv_linear} replaced with $1/(1+\min_{i\in\mc{V}}\frac{\eta_i\theta_i}{\beta_i+\frac{1-w_{ii}}{\alpha}})$.
\end{theorem}
\begin{proof}
See Appendix \ref{proof:ada_epsilon_0}.
\end{proof}

The step-size conditions \eqref{eq:ada_step-cond} and \eqref{eq:ada_step-cond_lowerbound} can be satisfied through a {\bf back-tracking} scheme: Choose $\eta_i,c_i\in(0,1)$ and $\gamma_i^{\text{init}}>0$, and set $\tilde{\gamma}_i=\gamma_i^{\text{init}}$. At each iteration $k\in\mc{K}_i$, 
    \begin{itemize}
        \item {\bf Step} 1: set $\gamma_i^k=\tilde{\gamma}_i$ and update $x_i^{k+1}$ by \eqref{eq:asyn_PB_alg}.
        \item {\bf Step} 2: compute $\beta_i^k$ according to \eqref{eq:beta_ik}.
        \item {\bf Step} 3: set $\tilde{\gamma}_i = \frac{c_i\eta_i}{\beta_i^k+\frac{1-w_{ii}}{\alpha}}$ and check \eqref{eq:ada_step-cond}. If \eqref{eq:ada_step-cond} holds, go to the next iteration; Otherwise, repeat {\bf Steps} 1-3.
    \end{itemize}
    Note that if \eqref{eq:ada_step-cond} fails to hold, then $\gamma_i^k>\frac{\eta_i}{\beta_i^k+\frac{1-w_{ii}}{\alpha}}$ and setting the new $\gamma_i^k= \frac{c_i\eta_i}{\beta_i^k+\frac{1-w_{ii}}{\alpha}}$ ensures $\gamma_i^k$ to be smaller than $c_i$ times of its old value. Moreover, because $\beta_i^k\le \beta_i$, as long as $\gamma_i^k\le \frac{\eta_i}{\beta_i+\frac{1-w_{ii}}{\alpha}}$, \eqref{eq:ada_step-cond} holds. Therefore, \eqref{eq:ada_step-cond} will be satisfied within a finite number of attempts. Moreover, following this scheme, it holds that for all $i\in\mc{V}$ and $k\in\mc{K}_i$,
    \begin{equation*}
        \gamma_i^k \ge \min\left(\frac{c_i\eta_i}{\beta_i+\frac{1-w_{ii}}{\alpha}}, \gamma_i^{\text{init}}\right),
    \end{equation*}
    which implies \eqref{eq:ada_step-cond_lowerbound}.

\begin{figure*}
    \begin{subfigure}[t]{\textwidth}  
    \hspace{0.2cm}
    \includegraphics[width=1\textwidth, height=1.3cm]{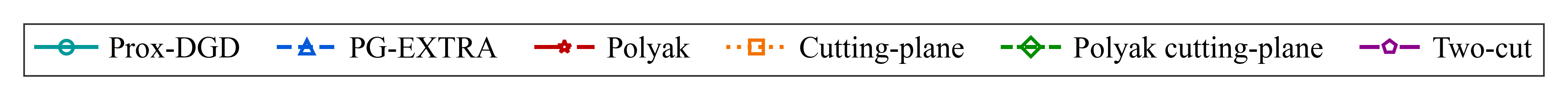}        
    \end{subfigure}
    \begin{minipage}[t]{\textwidth}
    \centering
    \begin{subfigure}[t]{0.33\textwidth}        
    \includegraphics[width=1\textwidth, valign=t]{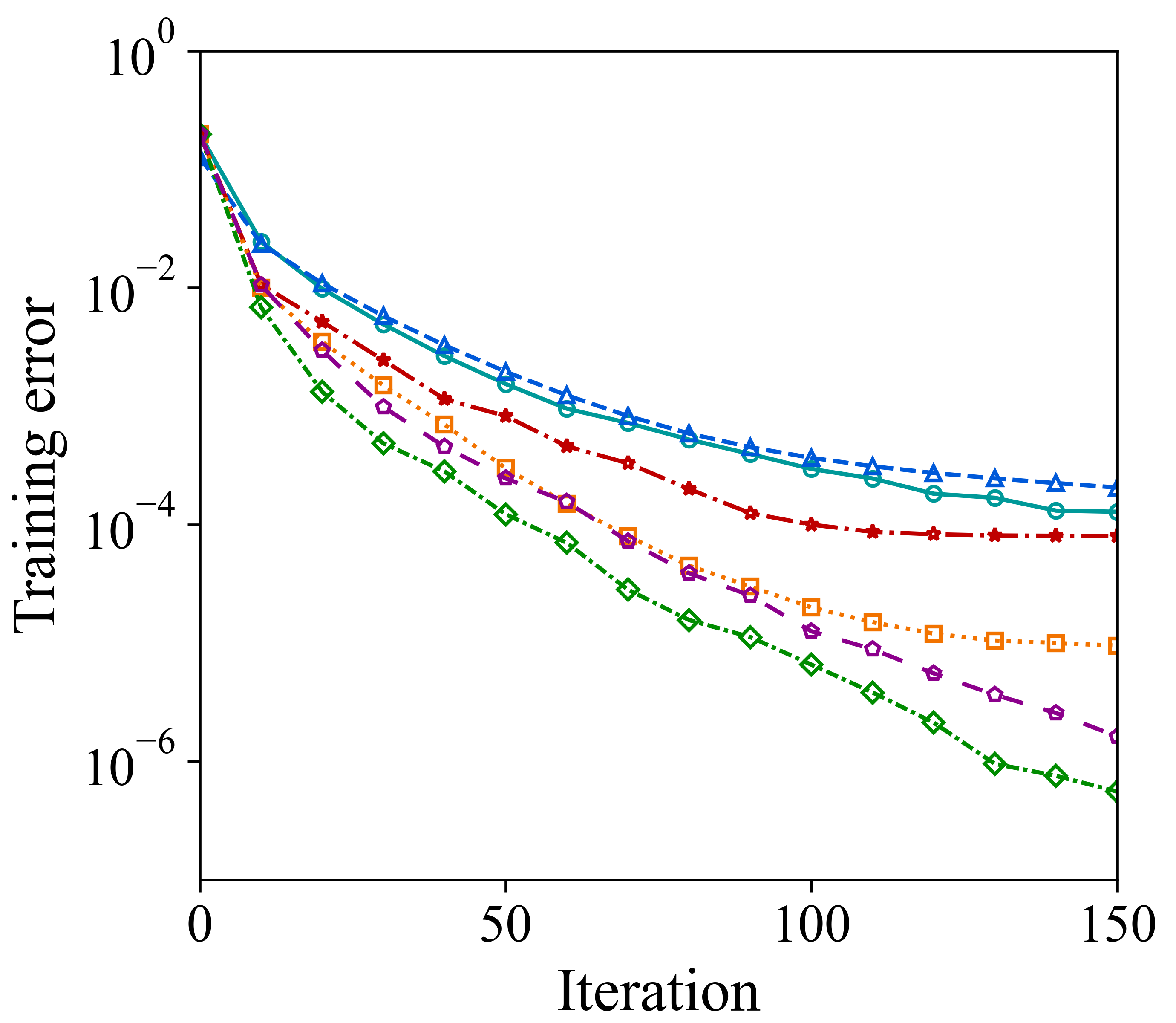}            
    \caption{Deterministic methods}        
    \label{fig:syn_algorithms_det_panel}    
    \end{subfigure}    
    \hspace{-0.3cm}    
    \begin{subfigure}[t]{0.33\textwidth}        
    \includegraphics[width=1\textwidth, valign=t]{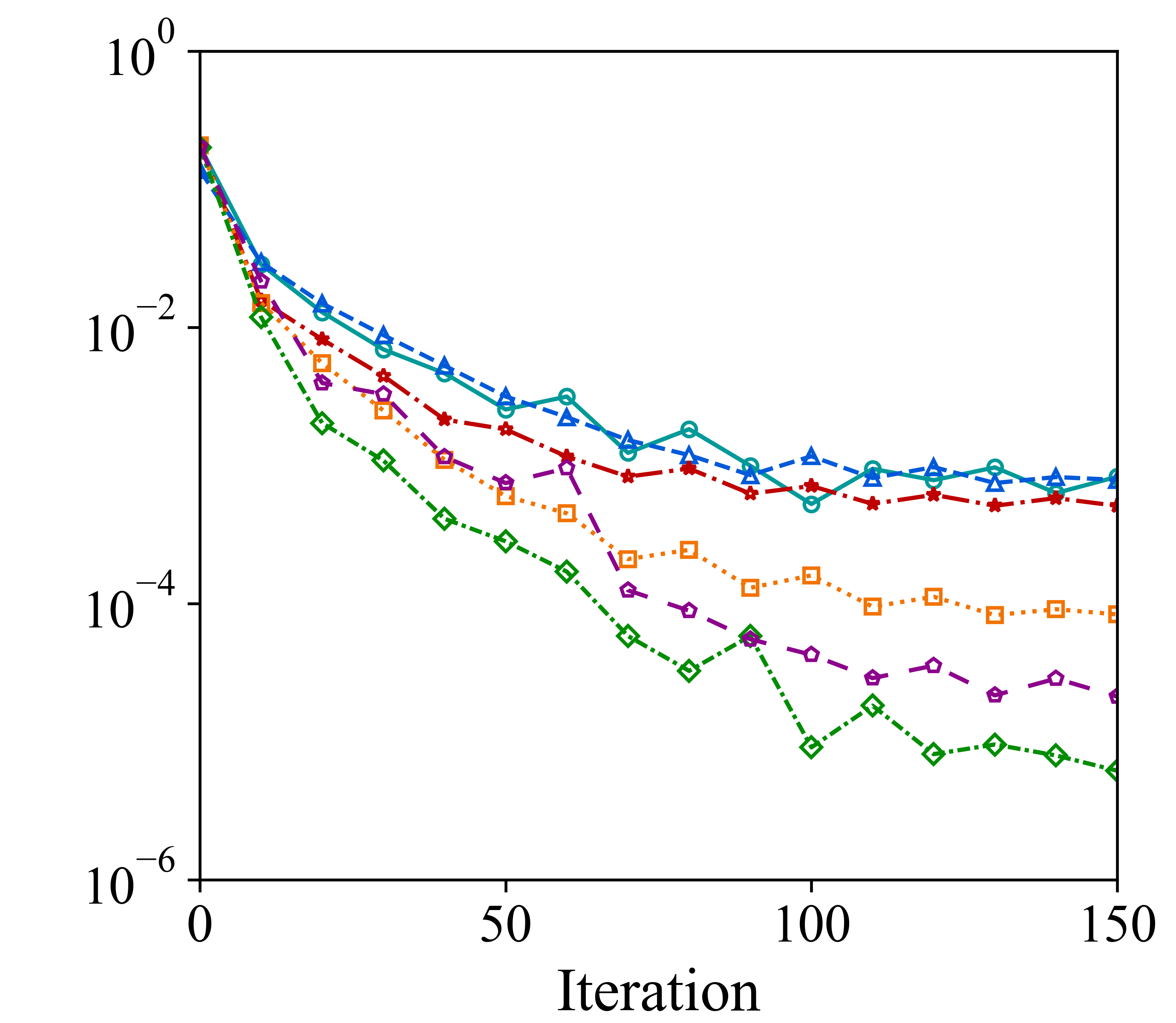}
    \caption{Stochastic methods}        
    \label{fig:syn_algorithms_sto_panel}    
    \end{subfigure}    
    \hspace{-0.3cm}    
    \begin{subfigure}[t]{0.33\textwidth}        
    \includegraphics[width=1\textwidth, valign=t]{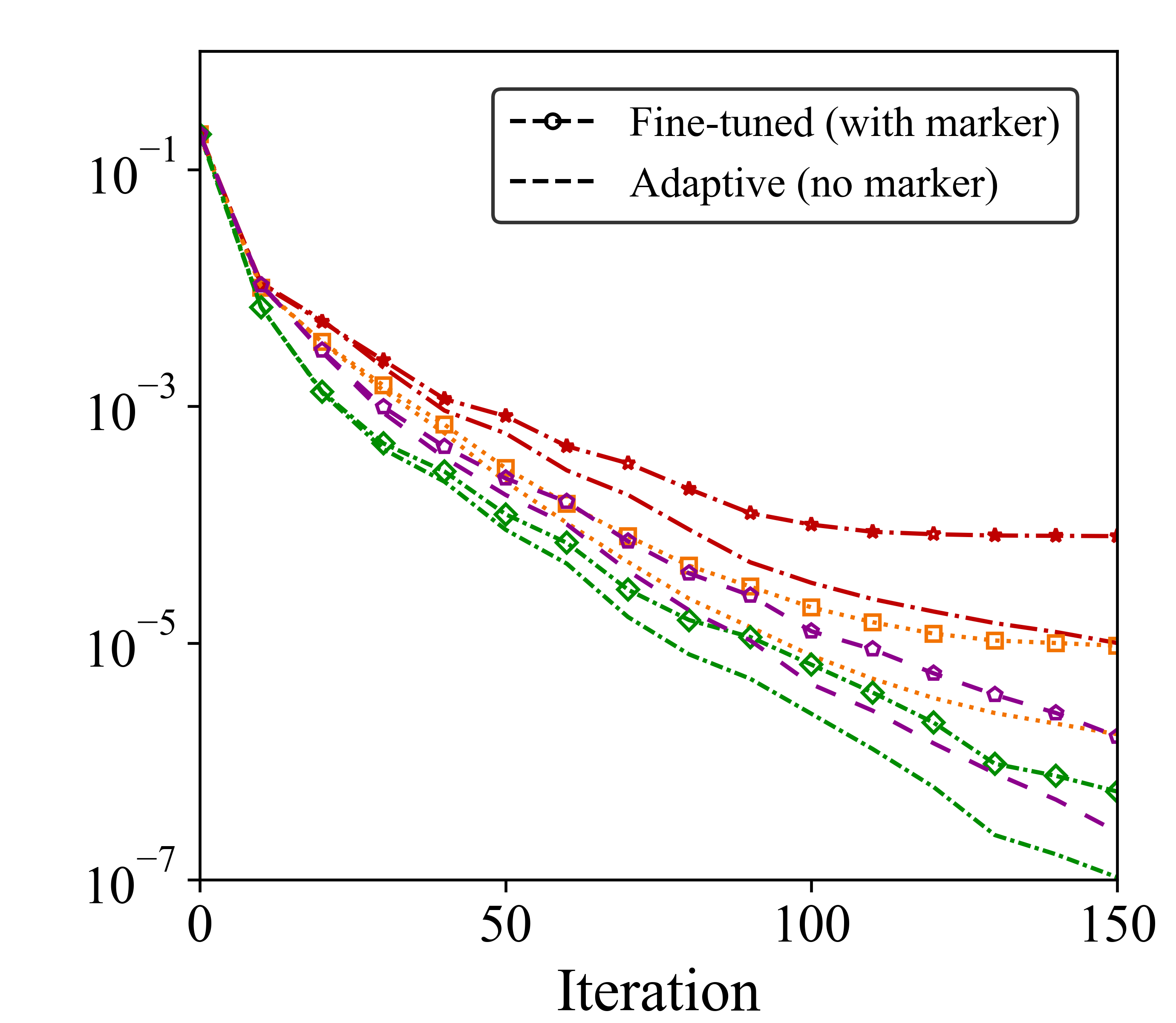}              
    \caption{Deterministic DPBM: fine-tuned \& adaptive}
    \label{fig:syn_step_size_comparison}    
    \end{subfigure}
    \end{minipage}
    \caption{Convergence of synchronous optimization methods.}
    \label{fig:syn_algorithm_comparison}
\end{figure*}

\begin{figure*}
    \begin{subfigure}[t]{\textwidth}  
    \hspace{0.2cm}
    \includegraphics[width=1\textwidth, height=1.3cm]{figures/algorithms_det_panel_legend_1.png}        
    \end{subfigure}
    \begin{minipage}[t]{\textwidth}
    \centering
    \begin{subfigure}[t]{0.33\textwidth}        
    \includegraphics[width=1\textwidth, valign=t]{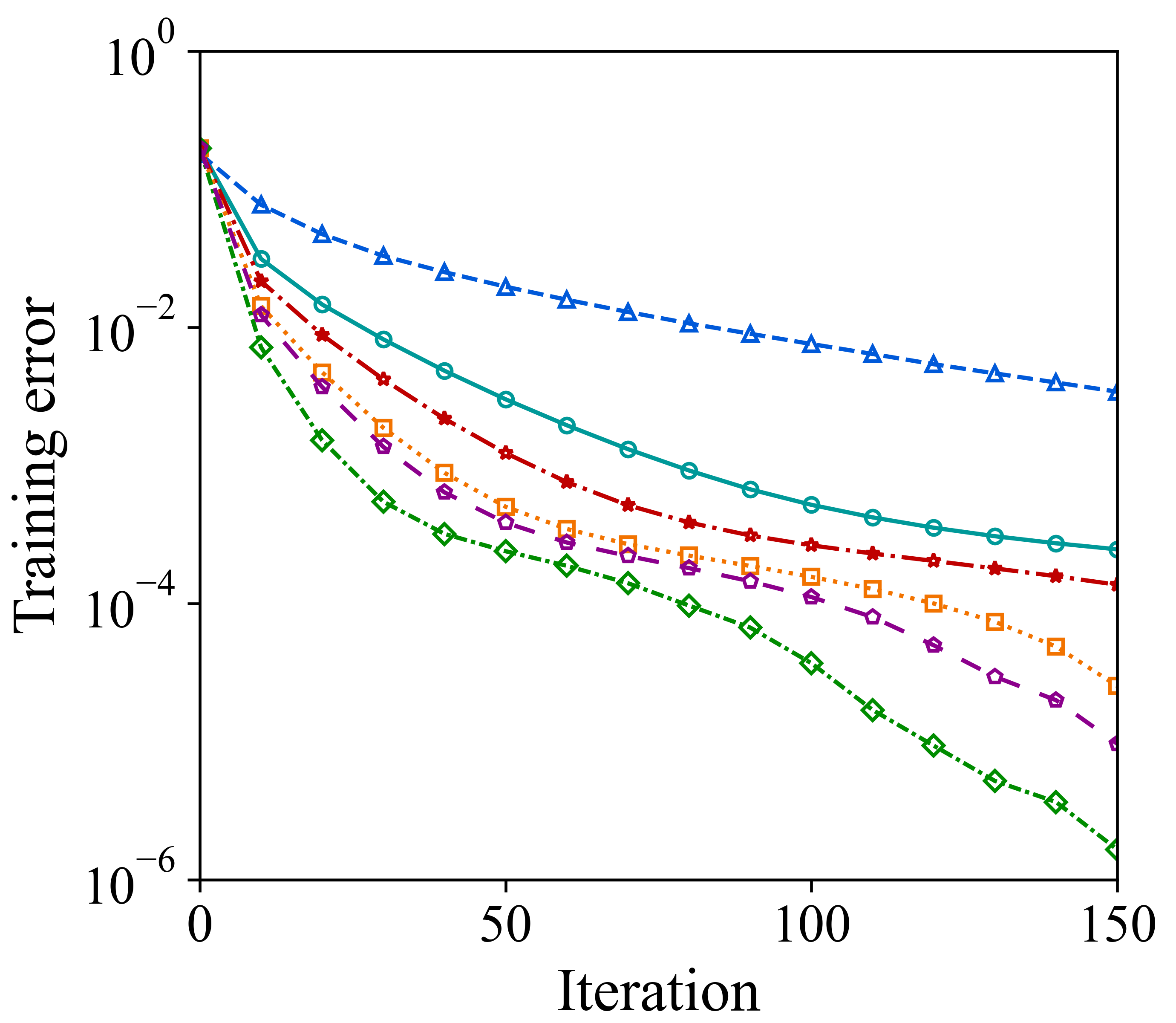}            
    \caption{Deterministic methods}        
    \label{fig:algorithms_det_panel}    
    \end{subfigure}    
    \hspace{-0.3cm}    
    \begin{subfigure}[t]{0.33\textwidth}        
    \includegraphics[width=1\textwidth, valign=t]{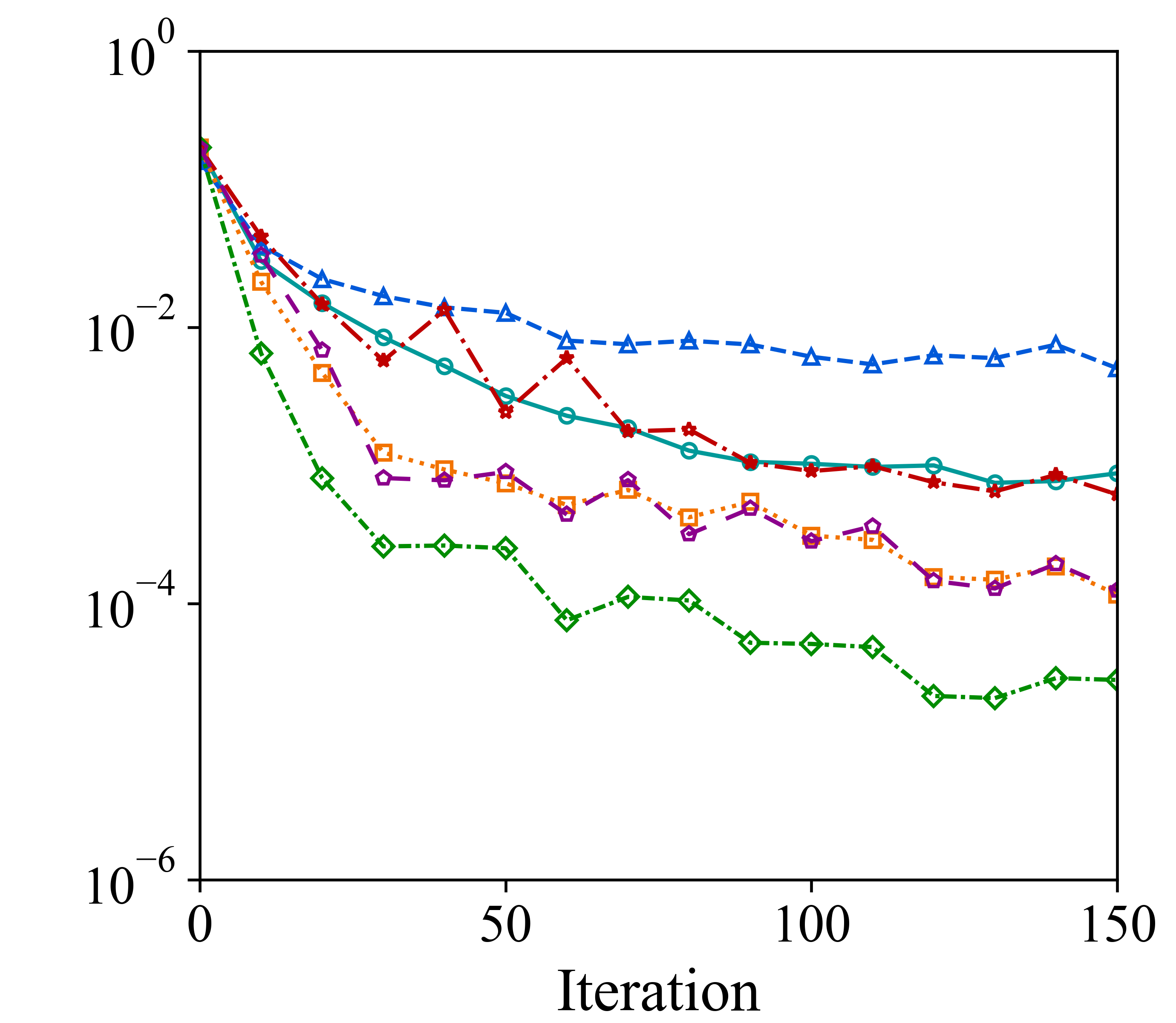}            
    \caption{Stochastic methods}        
    \label{fig:algorithms_sto_panel}    
    \end{subfigure}    
    \hspace{-0.3cm}    
    \begin{subfigure}[t]{0.33\textwidth}        
    \includegraphics[width=1\textwidth, valign=t]{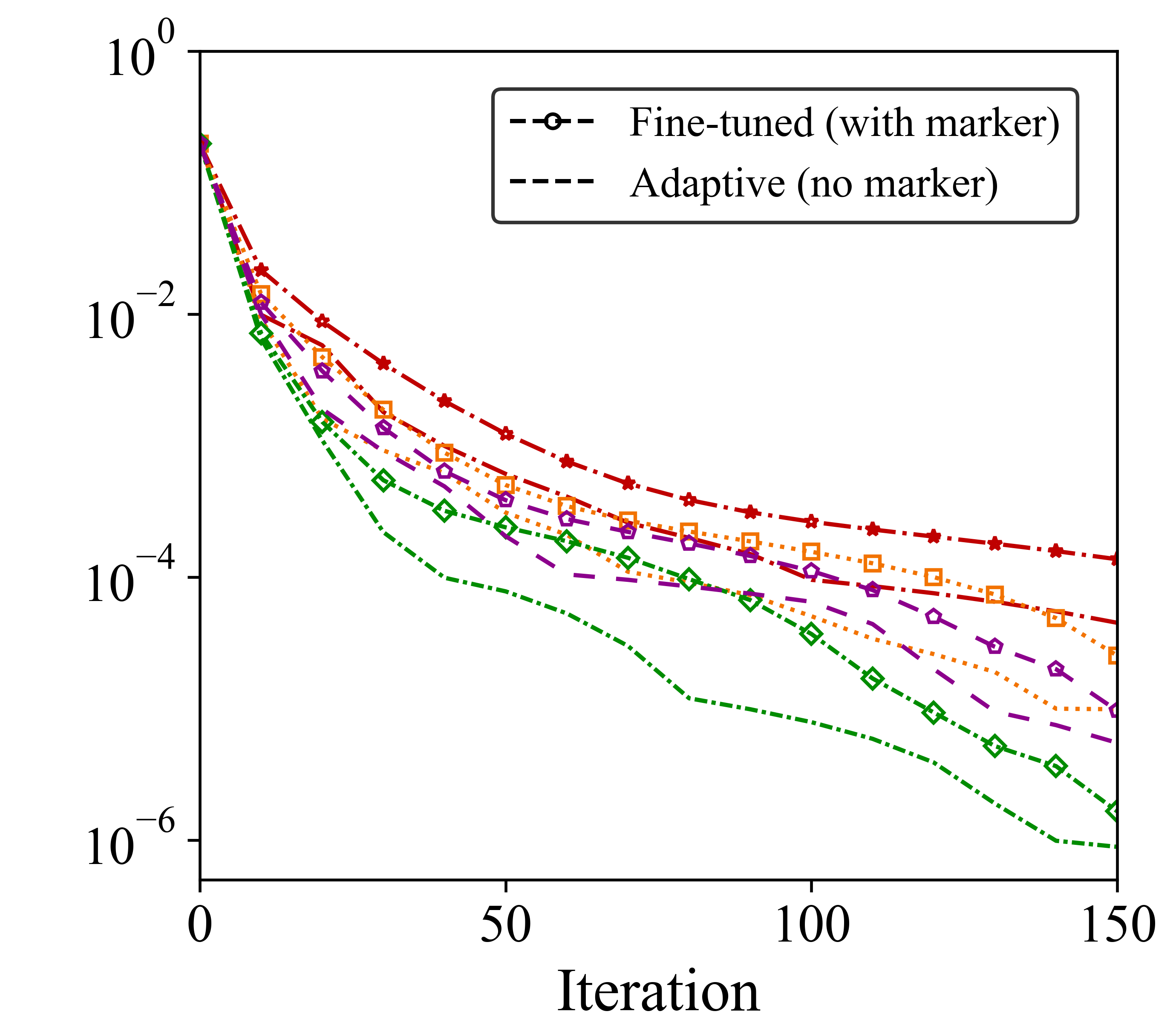}              
    \caption{Deterministic DPBM: fine-tuned \& adaptive}
    \label{fig:step_size_comparison}    
    \end{subfigure}
    \end{minipage}
    \caption{Convergence of asynchronous optimization methods.}
    \label{fig:algorithm_comparison}
\end{figure*}

\subsection{Stochastic Algorithm}

The asynchronous stochastic DPBM converges to a neighborhood of the optimum with delay-independent step-sizes.
\begin{theorem}\label{thm:sto_smooth}
    Suppose that Assumptions \ref{asm:convex}, \ref{asm:graph}, \ref{asm:opt_exist}, and \ref{asm:stoch_tilde_f} hold and each $\phi_i$ is $\theta_i$-strongly convex for some $\theta_i>0$. Also suppose that the step-size condition \eqref{eq:step-cond} holds. Then, for the update \eqref{eq:asyn_PB_alg}, under Assumption \ref{asm:partialasynchrony},
    \begin{equation}\label{eq:sto_conv}
        \mbb{E}[\|x_i^{k}-x_i^\star\|^2|~\mc{F}^k]\le \rho^{\lfloor k/(B+D+1)\rfloor}\max_{j\in\mc{V}}\|x_j^0-x_j^\star\|^2+C,
    \end{equation}
    where $C=\frac{2\rho}{1-\rho}\max_{i\in\mc{V}}(\epsilon_i+\epsilon_i^\star)$. Under Assumption \ref{asm:totalasynchrony}, \begin{equation}\label{eq:lim_sup}
        \underset{k\rightarrow+\infty}{\lim\sup}~\mbb{E}[\|x_i^{k}-x_i^\star\|^2|~\mc{F}^k]\le C.
    \end{equation}
\end{theorem}
\begin{proof}
See Appendix \ref{proof:sto_smooth}.
\end{proof}

\begin{remark}
    A closely related work \cite{zhou2025asynchronous} studies the convergence of the asynchronous stochastic decentralized gradient descent method, which is a special case of our bundle method with $m_i^k(x_i) = f_i(x_i^k)+\langle\nabla f_i(x_i^k), x_i-x_i^k\rangle$ and $h_i\equiv 0$. Under partial asynchrony and a step-size condition including $D$, the work \cite{zhou2025asynchronous} proves the convergence. Compared to \cite{zhou2025asynchronous}, our method 1) is more general, 2) can converge under delay-independent step-size conditions, 3) can handle non-smooth functions, and 4) is analysed under both the partial and the total asynchrony model. Admittedly, the work \cite{zhou2025asynchronous} allows for non-convex objective functions while our work does not.
\end{remark}

\section{Numerical experiments}\label{sec:numerical_exp}

We evaluate the empirical performance of the proposed methods on decentralized classification tasks. Each node $i$ possesses a set $\mathcal{D}_i$ of samples and all nodes collaborate to train a common classifier by logistic regression, which corresponds to problem \eqref{eq:prob} with
\begin{equation}
\begin{aligned}
    f_i(x) &= \frac{1}{|\mathcal{D}_i|} \sum_{(a_j,b_j)\in\mathcal{D}_i} \log\left(1 + e^{-b_j (a_j^T x)}\right),\\
    h_i(x) &= \lambda_1\|x\|_1.
\end{aligned}
\label{eq:experiment_setup}
\end{equation}
In problem \eqref{eq:experiment_setup}, \(a_j\) is $j$th feature vector and \(b_j\) is the corresponding label, and \(\lambda_1 = 10^{-3}\) is the regularization parameter. We set \(n = 20\) and use the ring network. The experiments are conducted on a multi-core computer, each node corresponds to a core, and all the delays are generated from real interactions between the cores rather than being manually set. We consider the Covertype dataset~\cite{Dua:2019} that contains \(581012\) samples with feature dimension $d=54$, and evenly distribute the samples to all nodes.

We compare DPBM with Prox-DGD \cite{zeng2018nonconvex,Wu25} and PG-EXTRA \cite{shi2015proximal,wu2017decentralized} in the synchronous and asynchronous, and deterministic and stochastic settings.  For stochastic methods, we set the batch size as \(100\). For DPBM, we set \(\alpha = 20\) and adopt four types of surrogate models for \(m_i^k\): the Polyak model, the cutting-plane model, the Polyak cutting-plane model, and the two-cut model, where $\mc{S}_i^k=[k-M+1,\ldots, k]$ in both the cutting-plane model and the Polyak cutting-plane model and $M=10$ by default. We run all methods for $150$ iterations and record the training error \( f(\bar{x}(t)) - f^* \) at the average iterate \(\bar{x}(t) = \frac{1}{n}\sum_{i=1}^n x_i(t)\), where \(x_i(t)\) is the value of \(x_i\) at the $t$th iteration and \(f^\star\) is the optimal value of (\ref{eq:prob}). Figs. \ref{fig:syn_algorithm_comparison} and \ref{fig:algorithm_comparison} compare the convergence of the synchronous and the asynchronous methods, respectively, and each of them considers the three settings: (a) deterministic methods with fine-tuned parameters; (b) stochastic methods with fine-tuned parameters; (c) deterministic DPBM with fine-tuned and the proposed adaptive step-size (below Theorem \ref{thm:epsilon_0_adapt}). Fig. \ref{fig:robust} demonstrates the robustness of DPBM with the Polyak cutting-plane model in the step-size $\gamma$ (we set $\gamma_i^k=\gamma$ for all $i\in\mc{V}$ and $k\in \N_0$) by showing the final training error after $150$ iterations.

From Figs. \ref{fig:syn_algorithm_comparison}--\ref{fig:robust}, we make the following observations. First, compared to Prox-DGD and PG-EXTRA, DPBM with the Polyak step-size yields comparable performance in all settings, while the remaining three models could achieve much faster convergence. This highlights the effectiveness of incorporating historical information in accelerating convergence (Prox-DGD corresponds to DPBM using current information only). Second, DPBM with the Polyak cutting-plane model consistently outperforms the Polyak model and the cutting-plane model, which may be due to the higher approximation accuracy of the Polyak cutting-plane model. Third, the adaptive step-size consistently leads to faster convergence than the fine-tuned step-size, while the latter requires no step-size tuning, which demonstrates the strong practical effectiveness of the proposed adaptive step-size. Fourth, increasing the number of cuts in DPBM—or equivalently, incorporating more historical information—significantly enhances its robustness in the step-size, e.g., in the synchronous and deterministic setting, DPBM with $M=1$ starts to diverge when $\gamma\ge 8$, while it with $M=10$ allows for $\gamma=20$. This strong robustness reduces the cost of parameter tuning.

\begin{figure*}
    \begin{subfigure}[t]{\textwidth}  
    \hspace{0.2cm}
    \centering
    \includegraphics[width=0.8\textwidth, height=1cm]{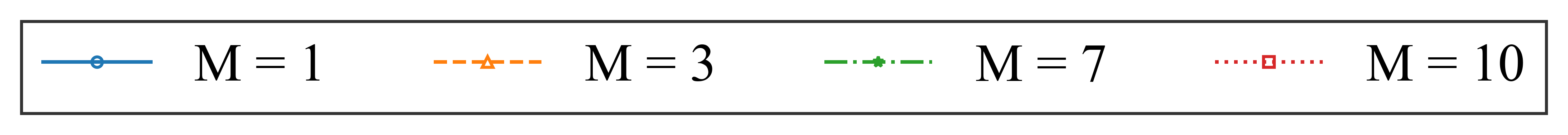}        
    \end{subfigure}
    \begin{minipage}[t]{\textwidth}
    \centering
    \begin{subfigure}[t]{0.33\textwidth}        
    \includegraphics[width=1\textwidth, valign=t]{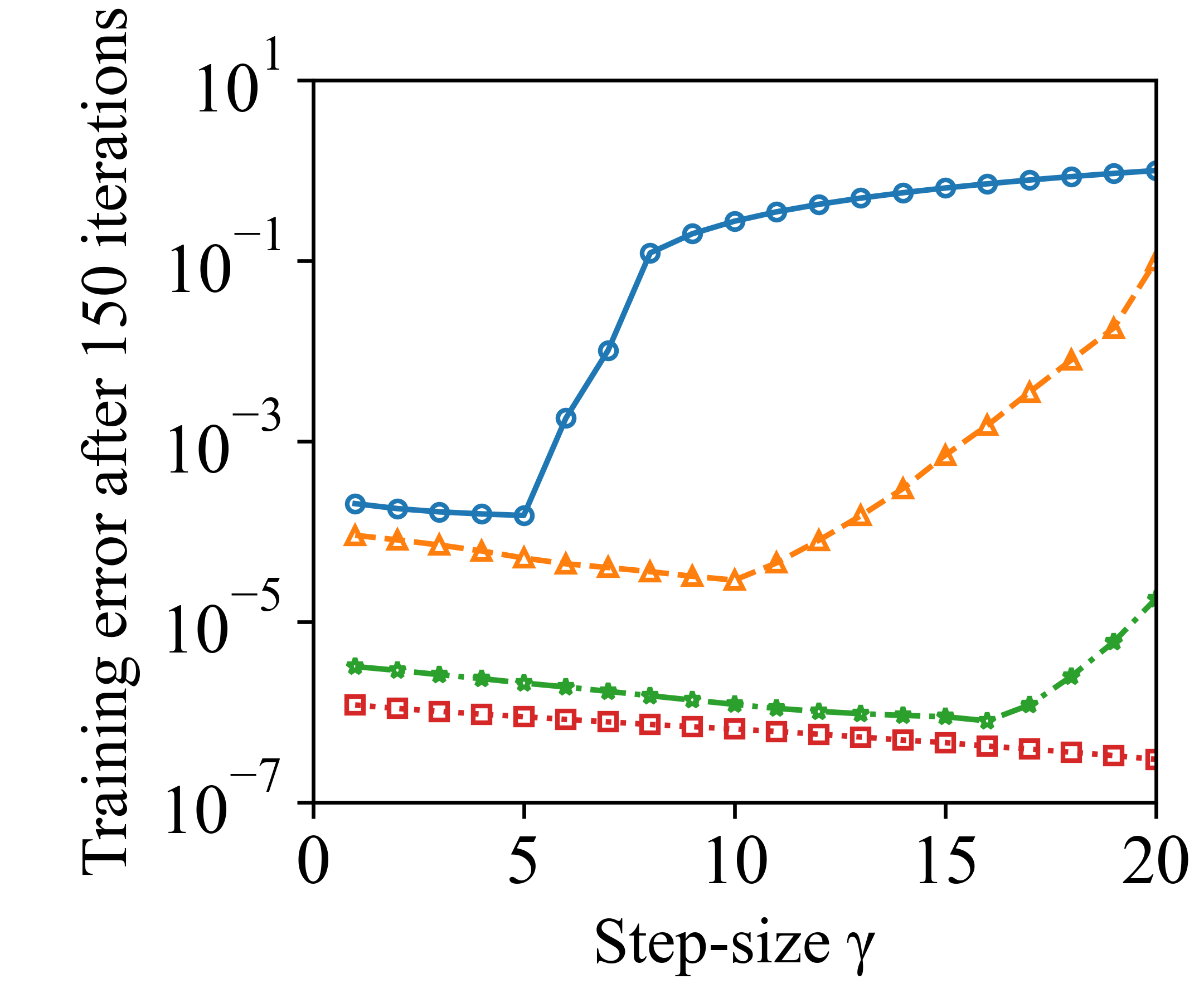}            
    \caption{Synchronous}        
    \label{fig:syn_det_gamma_bundle}    
    \end{subfigure}    
    \hspace{-0.3cm}    
    \begin{subfigure}[t]{0.33\textwidth}        
    \includegraphics[width=1\textwidth, valign=t]{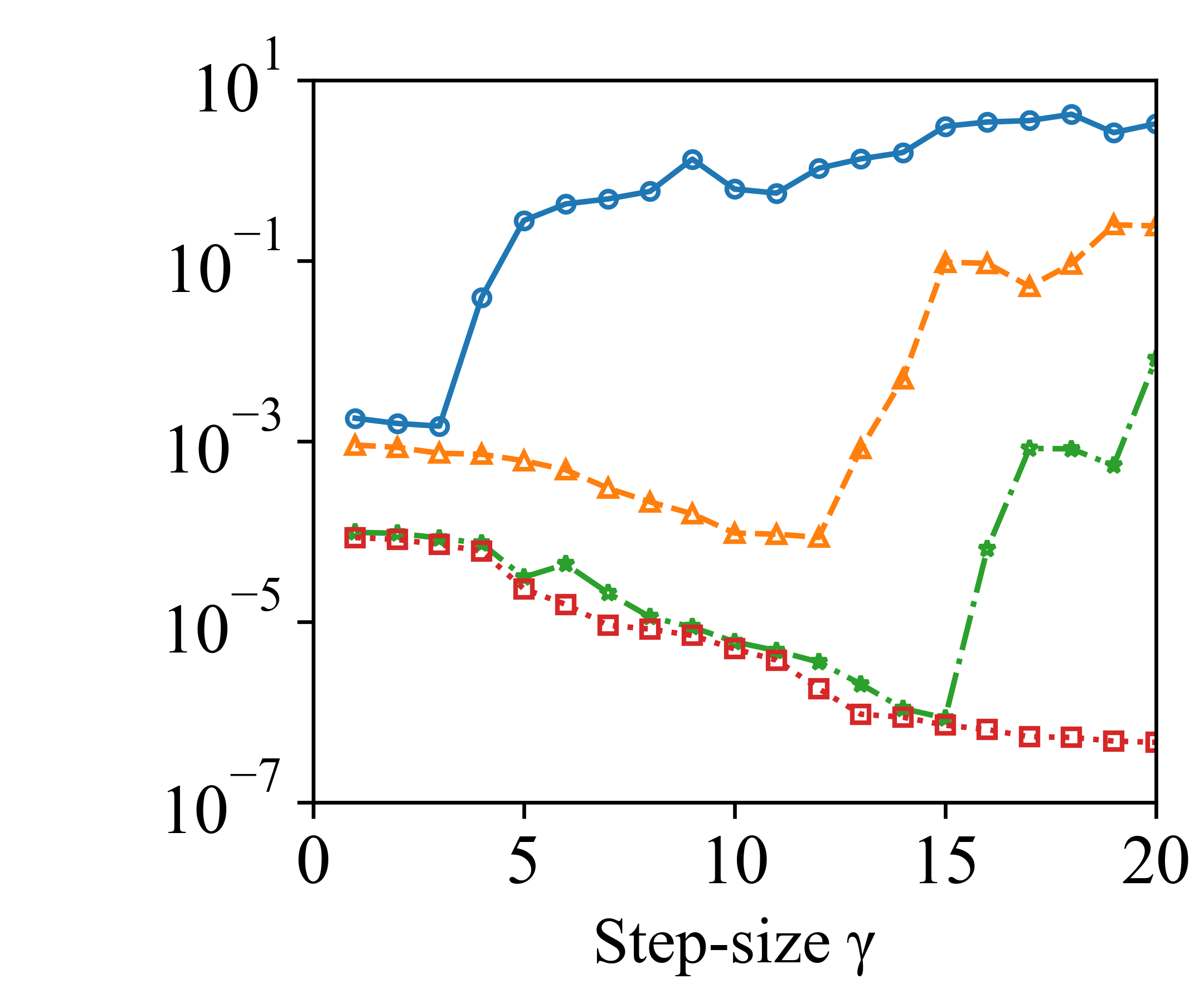}            
    \caption{Asynchronous}        
    \label{fig:asy_det_gamma_bundle}    
    \end{subfigure}    
    \hspace{-0.3cm}    
    \begin{subfigure}[t]{0.33\textwidth}        
    \includegraphics[width=1\textwidth, valign=t]{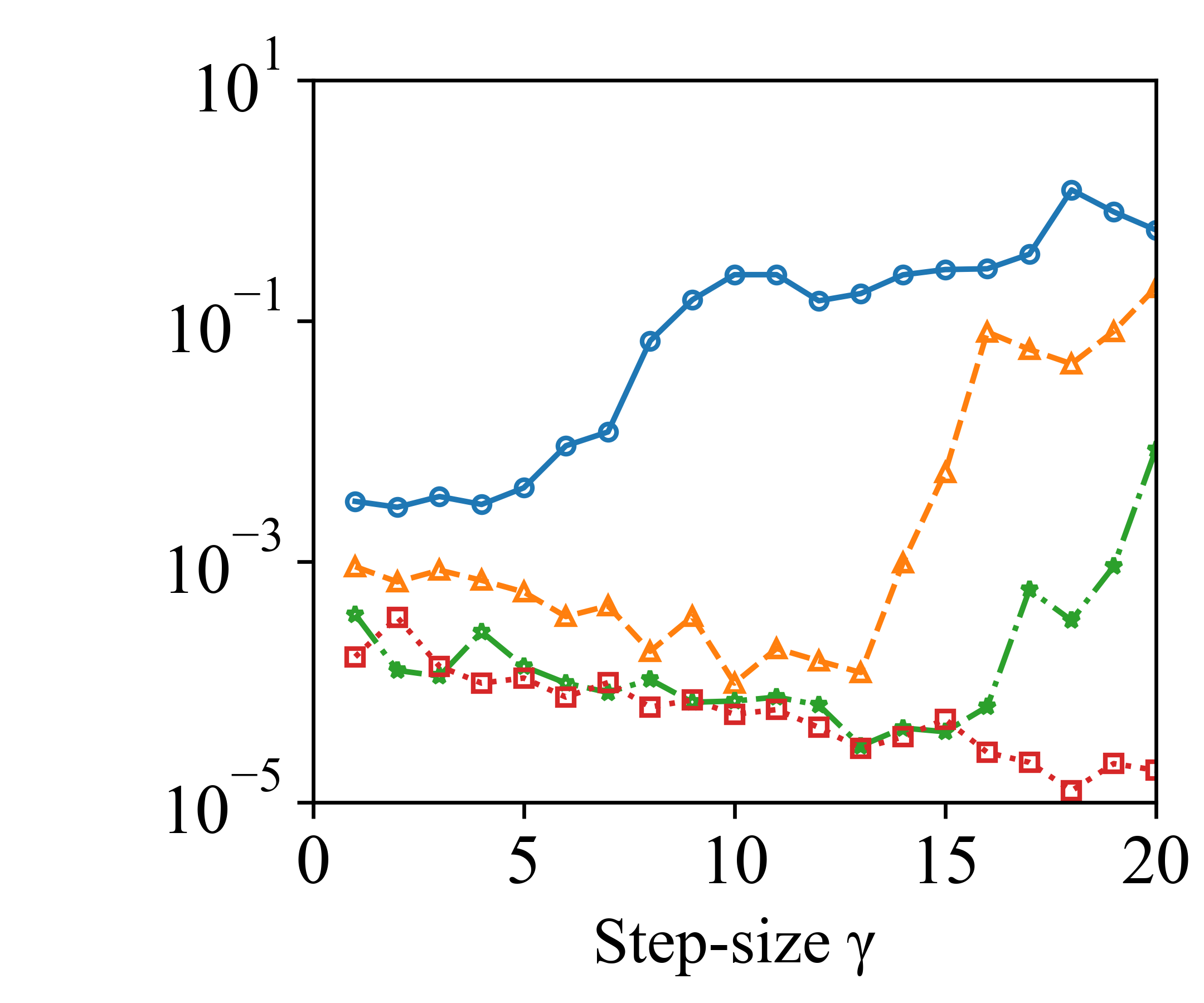}              
    \caption{Asynchronous stochastic}
    \label{fig:sto_gamma_bundle}    
    \end{subfigure}
    \end{minipage}
    \caption{Robustness of DPBM (with different cut number $M$) in the step-size $\gamma$.}
    \label{fig:robust}
\end{figure*}

\section{Conclusion}\label{sec:conclusion}

This paper leverages historical information to accelerate decentralized optimization methods. In particular, we apply the proximal bundle framework to the proximal decentralized gradient method (Prox-DGD), resulting in the Decentralized Proximal Bundle Method (DPBM). We further extend DPBM to asynchronous and stochastic settings to broaden its applicability. We analysed the convergence of the proposed methods under mild conditions. In the asynchronous setting, our methods are guaranteed to converge with delay-independent fixed step-sizes, which is generally less conservative and easier to determine compared to the delay-dependent step-size in most existing asynchronous optimization methods. The numerical results show that compared to the counterparts using only information at the current iterate, our methods not only converge faster, but also demonstrate strong robustness in the step-size. Future work will focus on accelerating decentralized exactly convergent algorithms with historical information.

\appendix
\subsection{Proof of Lemma \ref{lemma:det_key_prop}}\label{proof:lemma_det_key_prop}
To prove that Assumption \ref{asm:tilde_f} holds, it suffices to show the convexity of $m_i^k$,
\begin{align}
    & m_i^k(x_i)\le f_i(x_i),\label{eq:minorant}\\
    & f_i(x_i)\le m_i^k(x_i)+\frac{\beta_i}{2}\|x_i-x_i^k\|^2.\label{eq:det_mk_lower}
\end{align}
Since $m_i^k$ in \eqref{eq:pol}--\eqref{eq:two_cut} takes the maximum of affine functions, it is convex \cite[Section 3.2.3]{boyd2004convex}.

Next, we show \eqref{eq:minorant}. By the convexity of $f_i$ and $c_i^f\le \inf_{x_i} f_i(x_i)$, it is straightforward to see \eqref{eq:minorant} for $m_i^k$ in \eqref{eq:pol}--\eqref{eq:pol_cp}. For $m_i^k$ in \eqref{eq:two_cut}, we show \eqref{eq:minorant} by induction. First, by the convexity of $f_i$, we have \[m_i^0(x_i)=f_i(x_i^0)+\langle\nabla f_i(x_i^0), x_i-x_i^0\rangle\le f_i(x_i),\] i.e., \eqref{eq:minorant} holds at $k=0$. Fix $k\in \N_0$ and suppose that \eqref{eq:minorant} holds for all $t\le k$. If $k\notin \mc{K}_i$, then $m_i^{k+1}=m_i^k$ and, therefore, \eqref{eq:minorant} holds at $k+1$. If $k\in\mc{K}_i$, then by \eqref{eq:two_cut} and the convexity of $m_i^k$ and $f_i$, it holds that \[m_i^{k+1}(x_i)\le \max\{m_i^k(x_i), f_i(x_i)\}\le f_i(x_i),~\forall x_i\in\R^d,\]
i.e., \eqref{eq:minorant} holds at $k+1$. Concluding all the above, \eqref{eq:minorant} holds for all $k\in\N_0$.

Finally, we prove \eqref{eq:det_mk_lower}. For $m_i^k$ in \eqref{eq:pol}--\eqref{eq:two_cut},
\begin{equation}\label{eq:tilde_fik_largerthan_fik}
    m_i^k(x_i)\ge f_i(x_i^k) + \langle \nabla f_i(x_i^k), x_i-x_i^k\rangle.
\end{equation}
Since $f_i$ is $\beta_i$-smooth, we have
\[\begin{split}
    f_i(x_i^k) + \langle \nabla f_i(x_i^k), x_i-x_i^k\rangle + \frac{\beta_i}{2}\|x_i-x_i^k\|^2 \ge f_i(x_i).
\end{split}\]
Substituting the above equation into \eqref{eq:tilde_fik_largerthan_fik} yields \eqref{eq:det_mk_lower}.

\subsection{Proof of Lemma \ref{lemma:assum_sto}}\label{proof:assum_sto}
To prove that Assumption \ref{asm:stoch_tilde_f} holds, we need to show the convexity of $m_i^k$ and \eqref{eq:sto_smooth_tilde_sigma}--\eqref{eq:sto_smooth_tilde_sigma_star}. The convexity of $m_i^k$ is straightforward to see since the stochastic variants of the four models \eqref{eq:pol}--\eqref{eq:two_cut} are still piece-wise linear functions.

To show \eqref{eq:sto_smooth_tilde_sigma}, note that all the four models satisfy
\begin{equation*}
\begin{split}
    m_i^k(x_i) &\ge F_i(x_i^k;\mc{D}_i^k) + \langle \nabla F_i(x_i^k;\mc{D}_i^k), x_i-x_i^k\rangle\\
    &= f_i(x_i^k)+\langle \nabla f_i(x_i^k), x_i-x_i^k\rangle+\frac{L_i}{2}\|x_i-x_i^k\|^2\\
    &+F_i(x_i^k;\mc{D}_i^k)-f_i(x_i^k)-\frac{L_i}{2}\|x_i-x_i^k\|^2\\
    &+ \langle \nabla F_i(x_i^k;\mc{D}_i^k)-\nabla f_i(x_i^k), x_i-x_i^k\rangle,
\end{split}
\end{equation*}
in which
\begin{equation*}
    \begin{split}
        &\langle \nabla F_i(x_i^k;\mc{D}_i^k)-\nabla f_i(x_i^k), x_i-\nabla F_i(x_i^k;\mc{D}_i^k)\rangle\\
        \ge& -\frac{1}{2L_i}\|\nabla F_i(x_i^k;\mc{D}_i^k)-\nabla f_i(x_i^k)\|^2-\frac{L_i}{2}\|x_i-x_i^k\|^2.
    \end{split}
\end{equation*}
Moreover, since $F_i(x_i;\xi)$ is $L_i$-smooth for all $\xi\in\mc{D}_i$, so is $f_i$, which implies
\begin{equation*}
    f_i(x_i^k)+\langle \nabla f_i(x_i^k), x_i-x_i^k\rangle+\frac{L_i}{2}\|x_i-x_i^k\|^2\ge f_i(x_i).
\end{equation*}
Therefore,
\begin{equation*}
\begin{split}
    & m_i^k(x_i)+L_i\|x_i-x_i^k\|^2 - f_i(x_i)\\
    \ge& -\frac{1}{2L_i}\|\nabla F_i(x_i^k;\mc{D}_i^k)-\nabla f_i(x_i^k)\|^2+F_i(x_i^k;\mc{D}_i^k)-f_i(x_i^k).
\end{split}
\end{equation*}
Taking expectation on both sides over $\mc{D}_i^k$ and using \eqref{eq:sigma} yields \eqref{eq:sto_smooth_tilde_sigma} with $\beta_i=2L_i$ and $\epsilon_i = \frac{\sigma^2}{2L_i}$.

Finally, we prove \eqref{eq:sto_smooth_tilde_sigma_star}. For simplicity, define
\[\nu_i^t=F_i(x_i^t;\mc{D}_i^t) + \langle \nabla F_i(x_i^t;\mc{D}_i^t), x_i^\star-x_i^t\rangle, \forall t\in \N_0,i\in\mc{V}.\]
By the convexity of $F_i$, we have
\begin{equation}\label{eq:nut_bound}
    \nu_i^t \le F_i(x_i^\star; \mc{D}_i^t).
\end{equation}
For the stochastic variant of \eqref{eq:pol_cp},
\begin{equation}
\begin{split}
    & m_i^k(x_i^\star)-f_i(x_i^\star)\\
    =&\max\{\nu_i^t-f_i(x^\star),~t\in \mc{S}_i^k,c_i^f-f_i(x^\star)\}\\
    \le & \max\{F_i(x_i^\star; \mc{D}_i^t)-f_i(x^\star),~t\in \mc{S}_i^k\}\\
    \le& \sqrt{\sum_{t\in\mc{S}_i^k} (F_i(x_i^\star;\mc{D}_i^t)-f_i(x_i^\star))^2},
    \end{split}
\end{equation}
where the first inequality uses \eqref{eq:nut_bound} and $c_i^f\le f_i(x^\star)$. Taking expectation on both sides over $\mc{F}^k$ and using \eqref{eq:sigma_f} gives
\begin{equation}\label{eq:pol+CP}
\begin{split}
    &\mbb{E}[m_i^k(x_i^\star)|~\mc{F}^k]-f_i(x^\star)\\
    \le& \mbb{E}[\sqrt{\sum_{t\in\mc{S}_i^k} (F_i(x_i^\star;\mc{D}_i^t)-f_i(x_i^\star))^2}|~\mc{F}^k]\\
    \le& \sqrt{\sum_{t\in\mc{S}_i^k} \mbb{E}[(F_i(x_i^\star;\mc{D}_i^t)-f_i(x_i^\star))^2|~\mc{F}^k]}\\
    \le& \sqrt{|\mc{S}_i^k|}\sigma^\star.
\end{split}
\end{equation}
Since the stochastic varaint of \eqref{eq:pol} is a special case of that of \eqref{eq:pol_cp} with $\mc{S}_i^k=\{k\}$ and \eqref{eq:pol+CP} holds for the later, it follows that \eqref{eq:pol+CP} also holds for the stochastic variant of \eqref{eq:pol}. For the stochastic variant of \eqref{eq:cp}, following the derivation of \eqref{eq:pol+CP}, it is straightforward to see \eqref{eq:pol+CP} holds. Concluding all the above, \eqref{eq:sto_smooth_tilde_sigma_star} with $\epsilon_i^\star=\sqrt{\max_{k\in\mc{K}_i}|\mc{S}_i^k|}\sigma^\star$ holds for the stochastic variants of \eqref{eq:pol}--\eqref{eq:pol_cp}. Also note that for the stochastic variant of \eqref{eq:pol}, $|\mc{S}_i^k|=1$ for all $k\in\mc{K}_i$, which $\epsilon_i^\star=\sigma^\star$.

For the stochastic variant of \eqref{eq:two_cut}, the key inequality in the proof of \eqref{eq:sto_smooth_tilde_sigma_star} is
\begin{equation}\label{eq:mik-1}
    \begin{split}
        m_i^k(x_i^\star) \le f_i(x_i^\star)+\max_{\xi\in\mc{D}_i} (F_i(x_i^\star;\xi)-f_i(x_i^\star)),
    \end{split}
\end{equation}
which will be derived by induction. Clearly, \eqref{eq:mik-1} holds at $k=0$. Suppose that \eqref{eq:mik-1} holds at all $k\le K$. If $K\notin\mc{K}_i$, then we have $m_i^{K+1}=m_i^K$ and, therefore, \eqref{eq:mik-1} holds at $k=K+1$. Otherwise,
\[
\begin{split}
    &m_i^{K+1}(x_i^\star)-f_i(x_i^\star)\\
    \le& \max(m_i^{K}(x_i^\star)-f_i(x_i^\star), F_i(x_i^\star;\mc{D}_i^K)-f_i(x_i^\star))\\
    \le& \max_{\xi\in\mc{D}_i} (F_i(x_i^\star;\xi))-f_i(x_i^\star)),
\end{split}\]
where the first inequality uses the convexity of $F_i(\cdot;\mc{D}_i^K)$. Concluding all the above, \eqref{eq:mik-1} holds for all $i\in\mc{V}$ and $k\in\N_0$, which implies \eqref{eq:sto_smooth_tilde_sigma_star} with $\epsilon_i^\star = \max_{\xi\in\mc{D}_i} F_i(x_i^\star;\xi)-f_i(x_i^\star).$

\subsection{Proof of Theorem \ref{thm:epsilon_0}}\label{proof:epsilon_0}

Define $\bar{\mc{N}}_i=\mc{N}_i\cup\{i\}$ for all $i\in\mc{V}$ and
\begin{equation}\label{eq:hat_wij_def}
    \hat{w}_{ij} = \begin{cases}
        w_{ij}\gamma_i/\alpha, & j\ne i,\\
        1-(1-w_{ii})\gamma_i/\alpha, & \text{otherwise}.
    \end{cases}, \forall i, j\in\mc{V}.
\end{equation}
Note that $\hat{w}_{ij}>0$ if and only if $j\in\bar{\mc{N}}_i$ and $\sum_{j\in\mc{V}} \hat{w}_{ij} = 1$.

\begin{lemma}\label{lemma:smooth_deter}
    Suppose that all the conditions in Theorem \ref{thm:epsilon_0} hold. Under the step-size condition \eqref{eq:step-cond}, it holds that for some $g_i^\star\in\partial \phi_i(x_i^\star)$ and for all $i\in\mc{V}$ and $k\in\mc{K}_i$,
    \begin{equation}\label{eq:weak_contract}
    \begin{split}
        &\|x_i^{k+1}-x_i^\star\|^2\le \sum_{j\in \bar{\mc{N}}_i}\hat{w}_{ij}\|x_j^{s_{ij}^k}-x_j^\star\|^2\\
        &- 2\gamma_i(\phi_i(x_i^{k+1})-\phi_i(x_i^\star)-\langle g_i^\star, x_i^{k+1}-x_i^\star\rangle).
    \end{split}
    \end{equation}
\end{lemma}
\begin{proof}
    By the first-order optimality of \eqref{eq:asyn_PB_alg}, there exists $g_i^{k+1}\in \partial (m_i^k+h_i)(x_i^{k+1})$ such that 
    \begin{equation*}
        g_i^{k+1}+\frac{1}{\alpha}\left(x_i^k-\sum_{j\in \bar{\mc{N}}_i}w_{ij}x_j^{s_{ij}^k}\right)+\frac{1}{\gamma_i}(x_i^{k+1}-x_i^k) = 0,
    \end{equation*}
    which can be equivalently rewritten as \begin{equation}\label{eq:xi_k1}
        x_i^{k+1} = \sum_{j\in \bar{\mc{N}}_i}\hat{w}_{ij}x_j^{s_{ij}^k} - \gamma_i g_i^{k+1}.
    \end{equation}
    Moreover, since $\bx^\star$ is an optimum of \eqref{eq:penal_prob}, for a certain subgradient $g_i^\star\in\partial \phi_i(x_i^\star)$, we have
    \begin{equation}\label{eq:gistar_def}
        x_i^\star = \sum_{j\in\bar{\mc{N}}_i} w_{ij}x_j^\star - \alpha g_i^\star,
    \end{equation}
    so that
    \begin{equation}\label{eq:xi_star}
        \begin{split}
            x_i^\star &= x_i^\star+\frac{\gamma_i}{\alpha}\bigg(\sum_{j\in\bar{\mc{N}}_i} w_{ij}x_j^\star - \alpha g_i^\star - x_i^\star\bigg)\\
            &= \sum_{j\in\bar{\mc{N}}_i} \hat{w}_{ij}x_j^\star - \gamma_i g_i^\star.
        \end{split}
    \end{equation}
    By \eqref{eq:xi_k1} and \eqref{eq:xi_star},
    \begin{equation}\label{eq:xi_optexp}
        \begin{split}
        &\|x_i^{k+1}-x_i^\star\|^2\\
        = &\big\langle \sum_{j\in \bar{\mc{N}}_i}\hat{w}_{ij}x_j^{s_{ij}^k} - \gamma_i g_i^{k+1}-x_i^\star, x_i^{k+1}-x_i^\star\big\rangle\\
        = &\big\langle \sum_{j\in \bar{\mc{N}}_i}\hat{w}_{ij}(x_j^{s_{ij}^k} - x_j^\star)-\gamma_i g_i^{k+1}+\gamma_i g_i^\star, x_i^{k+1}-x_i^\star\big\rangle.
    \end{split}
    \end{equation}
    By the convexity of $m_i^k+h_i$ and $g_i^{k+1}\in \partial (m_i^k+h_i)(x_i^{k+1})$,
    \begin{equation}\label{eq:convexity}
        \begin{split}
            &-\gamma_i \langle g_i^{k+1}, x_i^{k+1}-x_i^\star\rangle\\
            \le &\gamma_i((m_i^k+h_i)(x_i^\star)-(m_i^k+h_i)(x_i^{k+1}))\\
            \le &\gamma_i(\phi_i(x_i^\star)-\phi_i(x_i^{k+1})+\frac{\beta_i}{2}\|x_i^{k+1}-x_i^k\|^2),
        \end{split}
    \end{equation}
    where the last step uses \eqref{eq:mk_key_prop}. Moreover,
    \begin{equation*}
    \begin{split}
        &\langle x_i^k-x_i^\star, x_i^{k+1}-x_i^\star\rangle\\
        \le& \frac{\|x_i^k-x_i^\star\|^2+\|x_i^{k+1}-x_i^\star\|^2-\|x_i^{k+1}-x_i^k\|^2}{2}
    \end{split}
    \end{equation*}
    and for each $j\in\mc{N}_i$,
    \begin{equation*}
    \begin{split}
        &\langle x_j^{s_{ij}^k}-x_j^\star, x_i^{k+1}-x_i^\star\rangle\\
        \le& \frac{\|x_j^{s_{ij}^k}-x_j^\star\|^2+\|x_i^{k+1}-x_i^\star\|^2}{2}.
    \end{split}
    \end{equation*}
    These, together with $s_{ii}^k=k$ and $\sum_{j\in\mc{V}} \hat{w}_{ij}=1$, give
    \begin{equation}\label{eq:cross_term}
        \begin{split}
            &\langle \sum_{j\in \bar{\mc{N}}_i}\hat{w}_{ij}(x_j^{s_{ij}^k}-x_j^\star), x_i^{k+1}-x_i^\star\rangle\\
            \le& \frac{\|x_i^{k+1}-x_i^\star\|^2}{2}+\sum_{j\in \bar{\mc{N}}_i}\frac{\hat{w}_{ij}\|x_j^{s_{ij}^k}-x_j^\star\|^2}{2}\\
            &-\frac{\hat{w}_{ii}\|x_i^{k+1}-x_i^k\|^2}{2}.
        \end{split}
    \end{equation}
    Substituting \eqref{eq:convexity} and \eqref{eq:cross_term} into \eqref{eq:xi_optexp} yields
    \begin{equation}\label{eq:uik1}
    \begin{split}
        &\|x_i^{k+1}-x_i^\star\|^2\le \sum_{j\in \bar{\mc{N}}_i}\hat{w}_{ij}\|x_j^{s_{ij}^k}-x_j^\star\|^2\\
        &-(\hat{w}_{ii}-\gamma_i\beta_i)\|x_i^{k+1}-x_i^k\|^2\\
        &- 2\gamma_i(\phi_i(x_i^{k+1})-\phi_i(x_i^\star)-\langle g_i^\star, x_i^{k+1}-x_i^\star\rangle).
    \end{split}
    \end{equation}
    In addition, by \eqref{eq:hat_wij_def} and \eqref{eq:step-cond}, we have
    \[1-\hat{w}_{ii} = (1-w_{ii})\gamma_i/\alpha < 1-\gamma_i\beta_i,\]
    which implies
    \begin{equation}\label{eq:gamma_i_smaller}
        \gamma_i < \hat{w}_{ii}/\beta_i.
    \end{equation}
    Substituting \eqref{eq:gamma_i_smaller} into \eqref{eq:uik1} yields \eqref{eq:weak_contract}. 
\end{proof}

Based on Lemma \ref{lemma:smooth_deter}, we prove the convergence of the algorithm under the partial asynchrony assumption and total asynchrony assumption, respectively. To this end, we introduce the block-wise maximum norm $\|\cdot\|_{b,\infty}$: For any $\bx\in\R^{nd}$,
\begin{equation}\label{eq:norm_def}
    \|\bx\|_{b,\infty} \overset{\triangle}{=} \max_{i\in [n]} \|x_i\|_2.
\end{equation}

\subsubsection{Proof under partial asynchrony (Assumption \ref{asm:partialasynchrony})}\label{sssec:partial}
Define $\tau=B+D$ and for any $t\in\N_0$,
\begin{align}
    \mc{I}_t &= [t(\tau+1), (t+1)(\tau+1)),\label{eq:It_def}\\
    V_i^t &= \max_{k\in\mc{I}_t}\|x_i^k-x_i^\star\|^2,~\forall i\in\mc{V},\label{eq:Vit_def}\\
    V^t &= \max_{i\in\mc{V}} V_i^t.\label{eq:Vt_def}
\end{align}
Our proof consists of {\bf five steps}. In particular, step 1 derives the existence of the limit
\begin{equation}\label{eq:Vstar_def}
    V^\star\overset{\triangle}{=}\lim_{t\rightarrow +\infty} V^t,
\end{equation}
step 2 shows that for all $i\in\mc{V}$,
\begin{equation}\label{eq:Vi_conv}
        \lim_{t\rightarrow +\infty} V_i^t = V^\star,
\end{equation}
step 3 proves that for all $i\in\mc{V}$,
\begin{equation}\label{eq:all_k_vstar}
    \lim_{k\rightarrow +\infty}\|x_i^k-x_i^\star\|^2=V^\star.
\end{equation}
Based on steps 1-3, step 4 establishes the convergence of the algorithm to an optimum of problem \eqref{eq:penal_prob}, and step 5 proves the linear convergence for strong convex $\phi_i$.

\vspace{0.2cm}

{\bf Step 1}: By Lemma \ref{lemma:smooth_deter} and the convexity of $\phi_i$, for any $i\in\mc{V}$ and $k\in\mc{K}_i$,
\begin{equation}\label{eq:copy_shrink}
    \begin{split}
        &\|x_i^{k+1}-x_i^\star\|^2\le \sum_{j\in\bar{\mc{N}}_i}\hat{w}_{ij}\|x_j^{s_{ij}^k}-x_j^\star\|^2.
    \end{split}
\end{equation}
Let \begin{equation}\label{eq:tik_def}
    t_i^k=\min\{t|~t<k, t\in\mc{K}_i\},~\forall i\in\mc{V},
\end{equation}
denote the last iteration that node $i$ updates. For any $k>\tau$ and $i\in\mc{V}$, by Assumption \ref{asm:partialasynchrony} we have $\mc{K}_i\cap [0,k) \ne \emptyset$ and
\begin{equation}\label{eq:tik_lb}
    t_i^k\ge k-B-1.
\end{equation}
Then, even if $k\notin \mc{K}_i$, we have $x_i^k = x_i^{t_i^k+1}$, which, together with \eqref{eq:copy_shrink}, yields
\begin{equation}\label{eq:xik1_minus_xistar}
\begin{split}
    \|x_i^k-x_i^\star\|^2&= \|x_i^{t_i^k+1}-x_i^\star\|^2\\
    &\le \sum_{j\in\bar{\mc{N}}_i}\hat{w}_{ij}\|x_j^{s_{ij}^{t_i^k}}-x_j^\star\|^2.
\end{split}
\end{equation}
Moreover, by \eqref{eq:sijk_range} and \eqref{eq:tik_lb}, 
\begin{equation}\label{eq:sij_range}
    s_{ij}^{t_i^k}\overset{\eqref{eq:sijk_range}}{\ge} t_i^k-D\overset{\eqref{eq:tik_lb}}{\ge} k-B-D-1=k-\tau-1.
\end{equation}
By \eqref{eq:xik1_minus_xistar}, \eqref{eq:sij_range}, and the definition of $\|\cdot\|_{b,\infty}$ in \eqref{eq:norm_def}, for any $k\ge \tau+1$,
\begin{equation*}
    \|\bx^k-\bx^\star\|_{b,\infty}^2\le \max_{k-\tau-1\le \ell\le k-1}\|\bx^\ell-\bx^\star\|_{b,\infty}^2,
\end{equation*}
which further leads to
\begin{equation}\label{eq:monotone_Vt}
    \begin{split}
        \underbrace{\max_{k\in\mc{I}_{t+1}}\|\bx^k-\bx^\star\|_{b,\infty}^2}_{V^{t+1}} \le \underbrace{\max_{k\in\mc{I}_{t}}\|\bx^k-\bx^\star\|_{b,\infty}^2}_{V^t}, \forall t\in\N_0.
    \end{split}
\end{equation}
Since $\{V^t\}$ is monotonically non-increasing and non-negative, the limit $V^\star$ exists.
\vspace{0.2cm}

{\bf Step 2:} We prove it by contradiction. If \eqref{eq:Vi_conv} does not hold for some $i\in\mc{V}$, then there exists $\epsilon'>0$ such that for any $T>0$, there exists $t'\ge T$ such that
\begin{equation}\label{eq:contra_result}
    |V_i^{t'}-V^\star|> \epsilon'.
\end{equation}
By \eqref{eq:Vstar_def}, for any $\epsilon>0$, there exists $T_\epsilon$ such that for all $t\ge T_\epsilon$,
\begin{equation}\label{eq:Vt_upper_bound}
    V^t\le V^\star + \epsilon.
\end{equation}
Let $\epsilon\le \epsilon'$, $T=T_\epsilon$, and $t' \ge T$ be such that \eqref{eq:contra_result} holds. Because $t'\ge T=T_\epsilon$, $V^{t'}\ge V_i^{t'}$ by \eqref{eq:Vt_def}, and $\epsilon\le \epsilon'$, we have
\begin{equation*}
    V_i^{t'}\le V^{t'} \le V^\star+\epsilon\le V^\star+\epsilon',
\end{equation*}
which, together with \eqref{eq:contra_result}, yields
\begin{equation}\label{eq:Vit_small}
    V_i^{t'} < V^\star-\epsilon'.
\end{equation}
Based on \eqref{eq:Vit_small}, we will show that for sufficiently small $\epsilon>0$,
\begin{equation}\label{eq:Vt'+n}
    V^{t'+n} <V^\star.
\end{equation}
However, by \eqref{eq:monotone_Vt} and \eqref{eq:Vstar_def},
\begin{equation}\label{eq:Vt_lowerbound}
    V^t\ge V^\star,\quad \forall t\in \N_0,
\end{equation}
which causes a contradiction and therefore, \eqref{eq:Vi_conv} holds.

To show \eqref{eq:Vt'+n}, define $\mc{A}^{t'}=\{i\}$ and
\[\mc{A}^{t+1} = \cup_{j\in \mc{A}^t} \bar{\mc{N}}_j,\quad\forall t\ge t'.\]
Also define $\bar{V}^{t'}=V^\star-\epsilon'$ and
\[\bar{V}^{t+1} = (1-\hat{w}_{\min}^{\tau+2})(V^\star+\epsilon)+\hat{w}_{\min}^{\tau+2}\bar{V}^t,~\forall t\ge t',\]
where $\hat{w}_{\min}=\min_{i\in\mc{V}, j\in\bar{\mc{N}}_i} \hat{w}_{ij}>0$. For $t=t'$, by \eqref{eq:Vit_small} and $\mc{A}^{t'}=\{i\}$, it holds that
\begin{equation}\label{eq:barV_t}
    V_j^t\le \bar{V}^t,~\forall j\in\mc{A}^t.
\end{equation}
Suppose that \eqref{eq:barV_t} holds at $t$ for some $t\ge t'$. Below, we consider two cases to show \eqref{eq:barV_t} at $t+1$.

\emph{Case 1}: $j\in\mc{A}^t$. Let $k\in\mc{I}^{t+1}$. If $j$ updates at the $(k-1)$th iteration, then by \eqref{eq:copy_shrink},
\begin{equation*}
    \|x_j^k-x_j^\star\|^2\le \hat{w}_{jj}\|x_j^{k-1}-x_j^\star\|^2+\sum_{\ell\in\mc{N}_j}\hat{w}_{j\ell}\|x_\ell^{s_{j\ell}^{t_j^k}}-x_\ell^\star\|^2.
\end{equation*}
Note that by \eqref{eq:sij_range}, we have that $s_{j\ell}^{t_j^k}\in \mc{I}_{t}\cup \mc{I}_{t+1}$, which, together with \eqref{eq:monotone_Vt}, \eqref{eq:Vt_upper_bound}, and $t\ge t'\ge T_\epsilon$, yields
\begin{equation}\label{eq:x_ell^s}
    \|x_\ell^{s_{j\ell}^{t_j^k}}-x_\ell^\star\|^2\le V^t\le V^\star+\epsilon.
\end{equation}
Therefore,
\begin{equation}\label{eq:xjk-xjstar}
    \|x_j^k-x_j^\star\|^2\le \hat{w}_{jj}\|x_j^{k-1}-x_j^\star\|^2+(1-\hat{w}_{jj})(V^\star+\epsilon).
\end{equation}
Moreover, because $j\in\mc{A}^t$ and \eqref{eq:barV_t} holds at $t$, we have that for $k=(t+1)(\tau+1)$,
\begin{equation*}
    \|x_j^{k-1}-x_j^\star\|^2\le \bar{V}^t,
\end{equation*}
substituting which into \eqref{eq:xjk-xjstar} yields that 
\[\|x_j^{(t+1)(\tau+1)}-x_j^\star\|^2\le \hat{w}_{jj}\bar{V}^t+(1-\hat{w}_{jj})(V^\star+\epsilon).\]
By the above equation, \eqref{eq:xjk-xjstar}, and $\hat{w}_{jj}\ge \hat{w}_{\min}$, we have that for any $k\in\mc{I}_{t+1}$ and $j\in\mc{A}^t$,
\begin{equation}\label{eq:xjk-xjstar_jinAt}
\begin{split}
    \|x_j^k-x_j^\star\|^2&\le \hat{w}_{jj}^{\tau+1}\bar{V}^t+(1-\hat{w}_{jj}^{\tau+1})(V^\star+\epsilon)\\
    &\le \hat{w}_{\min}^{\tau+1}\bar{V}^t+(1-\hat{w}_{\min}^{\tau+1})(V^\star+\epsilon).
\end{split}
\end{equation}
By \eqref{eq:xjk-xjstar_jinAt}, $V_j^{t+1}\le \bar{V}^{t+1}$ for $j\in\mc{A}^t$.

\emph{Case 2}: $j\in\mc{A}^{t+1}\setminus \mc{A}^t$. By the definition of $\mc{A}^{t+1}$, for some $j'\in\mc{A}^t$, we have that $j\in\mc{N}_{j'}$. If node $j$ update at the $(k-1)$th iteration for any $k\in\mc{I}^{t+1}$, then by \eqref{eq:xik1_minus_xistar},
\begin{equation}\label{eq:jinAt1_xhjxstar}
    \|x_j^k-x_j^\star\|^2\le \hat{w}_{jj'}\|x_{j'}^{s_{jj'}^{t_j^k}}-x_j^\star\|^2+\sum_{\ell\ne j',\ell\in\bar{\mc{N}}_j}\hat{w}_{j\ell}\|x_\ell^{s_{j\ell}^{t_j^k}}-x_\ell^\star\|^2.
\end{equation}
Because $s_{jj'}^{t_j^k}\in \mc{I}_t\cup\mc{I}_{t+1}$ and $j'\in\mc{A}^t$, by \eqref{eq:barV_t} and \eqref{eq:xjk-xjstar_jinAt}, we have
\begin{equation}\label{eq:jinAt1_xj'}
    \|x_{j'}^{s_{jj'}^{t_j^k}}-x_j^\star\|^2\le \hat{w}_{\min}^{\tau+1}\bar{V}^t+(1-\hat{w}_{\min}^{\tau+1})(V^\star+\epsilon).
\end{equation}
Moreover, for $\ell\ne j'$, by $s_{j\ell}^{t_j^k}\in \mc{I}_t\cup\mc{I}_{t+1}$, \eqref{eq:Vt_upper_bound}, and $t\ge t'\ge T_\epsilon$,
\begin{equation}\label{eq:jinAt1_xell}
    \|x_\ell^{s_{j\ell}^{t_j^k}}-x_\ell^\star\|^2\le \max(V^t, V^{t+1})\le V^\star+\epsilon.
\end{equation}
Substituting \eqref{eq:jinAt1_xj'} and \eqref{eq:jinAt1_xell} into \eqref{eq:jinAt1_xhjxstar} and using $\hat{w}_{jj'}\ge \hat{w}_{\min}$, we have
\begin{equation*}
\begin{split}
    \|x_j^k-x_j^\star\|^2\le & \hat{w}_{jj'}(\hat{w}_{\min}^{\tau+1}\bar{V}^t+(1-\hat{w}_{\min}^{\tau+1})(V^\star+\epsilon))\\
    &+(1-\hat{w}_{jj'})(V^\star+\epsilon)\le \bar{V}^{t+1}.
\end{split}
\end{equation*}
Concluding all the above, in either \emph{case 1} or \emph{case 2}, it holds that $\|x_j^{k+1}-x_j^\star\|^2\le \bar{V}^{t+1}$ for all $k\in\mc{I}_{t+1}$, i.e., \eqref{eq:barV_t} holds at $t+1$. By induction, we have \eqref{eq:barV_t} at all $t\ge t'$.
Since the network is connected, we have
\begin{equation}\label{eq:At'+n}
    \mc{A}^{t'+n}=\mc{V}.
\end{equation}
Moreover, since $\epsilon$ can be arbitrarily close to $0$ and $\bar{V}^{t'+n}<V^\star$ when setting $\epsilon=0$, we can choose sufficiently small $\epsilon>0$ such that $\bar{V}^{t'+n}<V^\star$, which, together with \eqref{eq:At'+n} and \eqref{eq:barV_t}, yields \eqref{eq:Vt'+n}.

{\bf Step 3}:
    By \eqref{eq:Vi_conv}, for any $\epsilon>0$, there exists $T_\epsilon>0$ such that for any $t\ge T_\epsilon$,
    \begin{equation}\label{eq:Vjstar_range}
        V^\star-\epsilon\le V_j^t\le V^\star+\epsilon,~\forall j\in\mc{V}
    \end{equation}
    which, together with the definition of $V_j^t$, implies that for any $k\ge K_\epsilon\overset{\triangle} {=}T_\epsilon(\tau+1)$,
    \begin{equation}\label{eq:small_error}
        \|x_j^k-x_j^\star\|^2\le V^\star+\epsilon,\quad\forall j\in\mc{V}.
    \end{equation}

    Next, we prove \eqref{eq:all_k_vstar} by contradiction. Suppose that \eqref{eq:all_k_vstar} does not hold for some $i\in\mc{V}$. Then, for some $\epsilon'>0$ and any $K$, there exists $k'\ge K$ such that
    \[\big|\|x_i^{k'}-x_i^\star\|^2-V^\star\big|> \epsilon'.\]
    Since $\epsilon$ can be arbirarily close to $0$, we let $\epsilon<\epsilon'$ and $K=K_\epsilon+\tau+1$, which, together with \eqref{eq:small_error} and $k'\ge K$, yields
    \begin{equation}\label{eq:xikxistarsmall}
        \|x_i^{k'}-x_i^\star\|^2 < V^\star - \epsilon'.
    \end{equation}
    By \eqref{eq:copy_shrink}, for any $\ell\ge 0$ such that $k'+\ell\in\mc{K}_i$,
    \begin{equation}\label{eq:xik'+ell+1}
        \|x_i^{k'+\ell+1}-x_i^\star\|^2\le \sum_{j\in\bar{\mc{N}}_i}\hat{w}_{ij}\|x_j^{s_{ij}^{t_i^{k'+\ell+1}}}-x_j^\star\|^2,
    \end{equation}
    and by \eqref{eq:sij_range} and $k'+\ell\in\mc{K}_i$,
    \begin{align}
        & s_{ij}^{t_i^{k'+\ell+1}}\ge k'+\ell+1-(\tau+1)\ge K-(\tau+1)\ge K_\epsilon\nonumber\\
        & s_{ii}^{t_i^{k'+\ell+1}} = t_i^{k'+\ell+1} = k'+\ell.\label{eq:sii}
    \end{align}
    Then, by \eqref{eq:small_error}, 
    \[\|x_j^{s_{ij}^{t_i^{k'+\ell+1}}}-x_j^\star\|^2\le V^\star+\epsilon\]
    substituting which into \eqref{eq:xik'+ell+1} and using \eqref{eq:sii} gives
    \begin{equation}\label{eq:xik'ell}
        \|x_i^{k'+\ell+1}-x_i^\star\|^2\le \hat{w}_{ii}\|x_i^{k'+\ell}-x_i^\star\|^2+(1-\hat{w}_{ii})(V^\star+\epsilon).
    \end{equation}
    Moreover, $\|x_i^{k'+\ell+1}-x_i^\star\|^2=\|x_i^{k'+\ell}-x_i^\star\|^2$ when $k'+\ell\notin \mc{K}_i$, which, together with \eqref{eq:xikxistarsmall} and \eqref{eq:xik'ell}, yields
    \begin{equation}\label{eq:xik'ell}
        \|x_i^{k'+\ell}-x_i^\star\|^2\le V^\star - \hat{w}_{ii}^{2(\tau+1)}\epsilon'+(1-\hat{w}_{ii}^{2(\tau+1)})\epsilon
    \end{equation}
    for all $\ell\in [1,2(\tau+1)]$. Suppose that $k'\in \mc{I}_{t'}$ for some $t'\ge 0$. Then, by \eqref{eq:xik'ell},
    \[V_i^{t'+1}\le V^\star - \hat{w}_{ii}^{2(\tau+1)}\epsilon'+(1-\hat{w}_{ii}^{2(\tau+1)})\epsilon.\]
    Since $\epsilon$ can be arbitrarily small, the right-hand side of the above equation can be strictly smaller than $V^\star$ for sufficiently small $\epsilon>0$, which contradicts \eqref{eq:Vjstar_range}.

    Concluding all the above, \eqref{eq:all_k_vstar} holds for all $i\in\mc{V}$.

\textbf{Step 4}: By the definition of $t_i^k$ in \eqref{eq:tik_def},
\[x_i^{k+1} = x_i^{t_i^{k+1}+1},~\forall i\in\mc{V}, k\in\N_0,\]
which, together with \eqref{eq:uik1}, yields
\begin{equation}\label{eq:uik1_2}
    \begin{split}
        &\|x_i^{k+1}-x_i^\star\|^2 = \|x_i^{t_i^{k+1}+1}-x_i^\star\|^2\\
        \le& \sum_{j\in \bar{\mc{N}}_i}\hat{w}_{ij}\|x_j^{s_{ij}^{t_i^{k+1}}}-x_j^\star\|^2\\
        &-(\hat{w}_{ii}-\gamma_i\beta_i)\|x_i^{k+1}-x_i^{t_i^{k+1}}\|^2\\
        &+2\gamma_i(\phi_i(x_i^\star)-\phi_i(x_i^{k+1})-\langle g_i^\star, x_i^\star-x_i^{k+1}\rangle).
    \end{split}
    \end{equation}
    Also note that
    \begin{equation*}
        \|x_i^{k+1}-x_i^k\|^2 =\begin{cases}
            \|x_i^{k+1}-x_i^{t_i^{k+1}}\|^2, & t_i^{k+1}=k,\\
            0, & \text{otherwise.}
        \end{cases}
    \end{equation*}
    Therefore,
    \begin{equation*}
        \|x_i^{k+1}-x_i^k\|^2 \le \|x_i^{k+1}-x_i^{t_i^{k+1}}\|^2,
    \end{equation*}
    substituting which into \eqref{eq:uik1_2} gives
    \begin{equation}\label{eq:uik1_3}
    \begin{split}
        &\|x_i^{k+1}-x_i^\star\|^2\le \sum_{j\in \bar{\mc{N}}_i}\hat{w}_{ij}\|x_j^{s_{ij}^{t_i^{k+1}}}-x_j^\star\|^2\\
        &-(\hat{w}_{ii}-\gamma_i\beta_i)\|x_i^{k+1}-x_i^k\|^2\\
        &+2\gamma_i(\phi_i(x_i^\star)-\phi_i(x_i^{k+1})-\langle g_i^\star, x_i^\star-x_i^{k+1}\rangle).
    \end{split}
    \end{equation}
    By the convexity of $\phi_i$,
    \begin{equation}\label{eq:convexity_phi}
        \phi_i(x_i^\star)-\phi_i(x_i^{k+1})-\langle g_i^\star, x_i^\star-x_i^{k+1}\rangle\le 0.
    \end{equation}
    Taking limit on both sides of \eqref{eq:uik1_3} and using \eqref{eq:all_k_vstar}, \eqref{eq:convexity_phi}, and \eqref{eq:gamma_i_smaller} gives that for all $i\in\mc{V}$,
    \begin{align}
        &\lim_{k\rightarrow+\infty} \|x_i^{k+1}-x_i^k\|^2 = 0,\label{eq:xk1xkdiff}\\
        &\lim_{k\rightarrow+\infty} \phi_i(x_i^{k+1})-\phi_i(x_i^\star)-\langle g_i^\star, x_i^{k+1}-x_i^\star\rangle = 0.\label{eq:phi_conv}
    \end{align}
    By taking a look at the derivation from \eqref{eq:xi_optexp} to \eqref{eq:uik1} and using \eqref{eq:convexity_phi} and the notation $t_i^k$, we have that for all $i\in\mc{V}$ and $k\in\N_0$,
    \begin{equation}\label{eq:comp}
    \begin{split}
        &\|x_i^{k+1}-x_i^\star\|^2\\
        \le&\langle \sum_{j\in \bar{\mc{N}}_i}\hat{w}_{ij}(x_j^{s_{ij}^{t_i^{k+1}}}\!\!-x_j^\star), x_i^{k+1}-x_i^\star\rangle+\gamma_i\beta_i\|x_i^{k+1}-x_i^k\|^2/2\\
        \le& \langle \sum_{j\in \bar{\mc{N}}_i}\hat{w}_{ij}(x_j^k-x_j^\star), x_i^{k+1}-x_i^\star\rangle+R_i^k\\
        \le& \|\sum_{j\in \bar{\mc{N}}_i}\hat{w}_{ij} (x_j^k-x_j^\star)\|\cdot\|x_i^{k+1}-x_i^\star\|+R_i^k\\
        \le& \sum_{j\in\bar{\mc{N}}_i}\hat{w}_{ij} \|x_j^k-x_j^\star\|\cdot\|x_i^{k+1}-x_i^\star\|+R_i^k,
    \end{split}
    \end{equation}
    where $R_i^k=\sum_{j\in\bar{\mc{N}}_i} \hat{w}_{ij}\langle (x_j^{s_{ij}^{t_i^{k+1}}} - x_j^k), x_i^{k+1}-x_i^\star\rangle+\gamma_i\beta_i\|x_i^{k+1}-x_i^k\|^2/2$ converges to $0$ due to \eqref{eq:sij_range}, \eqref{eq:all_k_vstar}, and \eqref{eq:xk1xkdiff}.
    Taking limit on both sides of \eqref{eq:comp} and using \eqref{eq:all_k_vstar}, we have that the limit of both sides are $V^\star$. By the squeeze theorem for sequences and $\lim_{k\rightarrow+\infty} R_i^k=0$, we have
    \begin{equation}
    \begin{split}
        &\lim_{k\rightarrow +\infty} \|\sum_{j\in \bar{\mc{N}}_i}\hat{w}_{ij} (x_j^k-x_j^\star)\|\cdot \|x_i^{k+1}-x_i^\star\|=V^\star,
    \end{split}
    \end{equation}
    which, together with \eqref{eq:all_k_vstar}, indicates
    \begin{equation}\label{eq:conv_sumwijx}
        \lim_{k\rightarrow +\infty} \|\sum_{j\in \bar{\mc{N}}_i}\hat{w}_{ij} (x_j^k-x_j^\star)\| = \sqrt{V^\star}.
    \end{equation}
Note that by \eqref{eq:Vi_conv}, the sequence $\{\bx^k\}$ is bounded and therefore, has a limit point $\tilde{\bx}$ for a subsequence. Then, by \eqref{eq:all_k_vstar} and \eqref{eq:conv_sumwijx},
\begin{equation*}
    \sqrt{V^\star}=\|\sum_{j\in \bar{\mc{N}}_i}\hat{w}_{ij} (\tilde{x}_j-x_j^\star)\| = \sum_{j\in \bar{\mc{N}}_i}\hat{w}_{ij}\|\tilde{x}_j-x_j^\star\| = \sqrt{V^\star}.
\end{equation*}
By \cite[Lemma 5]{Wu25}, the above equation implies
    \begin{equation}
        \tilde{x}_i - x_i^\star = \tilde{x}_j-x_j^\star,\quad\forall j\in\mc{N}_i,
    \end{equation}
Moreover, $\mc{G}$ is connected. Then, we have
    \begin{equation}\label{eq:IminusWtilde_x}
        (I-\mb{W})(\tilde{\bx}-\bx^\star) = 0.
    \end{equation}
By \eqref{eq:gistar_def},
\[g^\star = \frac{(\mb{W}-I)\bx^\star}{\alpha},\]
which, together with \eqref{eq:IminusWtilde_x}, gives
\begin{equation*}
    \begin{split}
        & -\langle g^\star, \bx^\star-\tilde{\bx}\rangle\\
        =& \frac{\langle (I-\mb{W})\bx^\star, \bx^\star-\tilde{\bx}\rangle}{\alpha}\\
        =&\frac{1}{2\alpha}((\bx^\star)^T(I-\mb{W})\bx^\star-\tilde{\bx}^T(I-\mb{W})\tilde{\bx})\\
        &+\frac{1}{2\alpha}(\bx^\star-\tilde{\bx})^T(I-\mb{W})(\bx^\star-\tilde{\bx})\\
        =&\frac{1}{2\alpha}((\bx^\star)^T(I-\mb{W})\bx^\star-\tilde{\bx}^T(I-\mb{W})\tilde{\bx}).
    \end{split}
\end{equation*}
Moreover, by \eqref{eq:phi_conv},
\begin{equation}\label{eq:lim_phi_0}
    \phi(\bx^\star) - \phi(\tilde{\bx}) - \langle g^\star, \bx^\star-\tilde{\bx}\rangle = 0.
\end{equation}
Therefore,
\[\phi(\tilde{\bx})+\frac{1}{2\alpha}\tilde{\bx}^T(I-\mb{W})\tilde{\bx} = \phi(\bx^\star)+\frac{1}{2\alpha}(\bx^\star)^T(I-\mb{W})\bx^\star,\]
i.e., the objective value of the problem \eqref{eq:penal_prob} at $\tilde{\bx}$ equals to its optimal value, which indicates the optimality of $\tilde{\bx}$ to problem \eqref{eq:penal_prob}.
Moreover, by \eqref{eq:all_k_vstar}, we have
\[\|\tilde{x}_i-x_i^\star\| = V^\star,~\forall i\in\mc{V}.\]
Since $\bx^\star$ is an arbitrary optimum of \eqref{eq:penal_prob}, if we set $\bx^\star=\tilde{\bx}$, then $V^\star=0$, which, together with \eqref{eq:all_k_vstar}, indicates the convergence of $\{\bx^k\}$ to $\tilde{\bx}$.

Concluding all the above, the algorithm converges to an optimal solution of \eqref{eq:penal_prob}.

{\bf Step 5}: Since $\phi_i$ is $\theta_i$-strongly convex, for any $i\in\mc{V}$ and $k\in\N_0$,
\begin{equation}\label{eq:sc}
    \phi_i(x_i^{k+1})\!-\!\phi_i(x_i^\star)\!-\!\langle g_i^\star, x_i^{k+1}-x_i^\star\rangle\ge \frac{\theta_i\|x_i^{k+1}-x_i^\star\|^2}{2}.
\end{equation}
By \eqref{eq:uik1_3}, \eqref{eq:sc}, and \eqref{eq:gamma_i_smaller},
\begin{equation}\label{eq:uik1_4}
    \begin{split}
        \|x_i^{k+1}-x_i^\star\|^2\le &\frac{1}{1+\gamma_i\theta_i}\sum_{j\in \mc{V}}\hat{w}_{ij}\|x_j^{s_{ij}^{t_i^{k+1}}}-x_j^\star\|^2\\
        \overset{\eqref{eq:sij_range}}{\le} & \frac{1}{1+\gamma_i\theta_i}\max_{k-\tau\le \ell\le k}\|\bx^\ell-\bx^\star\|_{b,\infty}^2,
    \end{split}
\end{equation}
so that
\begin{equation}\label{eq:uik1_5}
    \begin{split}
        \|\bx^{k+1}-\bx^\star\|_{b,\infty}^2\le &\frac{\max_{k-\tau\le \ell\le k}\|\bx^\ell-\bx^\star\|_{b,\infty}^2}{1+\min_{i\in\mc{V}}\gamma_i\theta_i}.
    \end{split}
\end{equation}
Based on \eqref{eq:uik1_5}, it is straightforward to derive
\begin{equation}\label{eq:Vt_linear}
    V^t\le \rho^tV^0,
\end{equation}
where $V^t$ is defined in \eqref{eq:Vt_def} and $\rho=1/(1+\min_{i\in\mc{V}}\gamma_i\theta_i)$. For any $k\in\N_0$, since $k\in \mc{I}_{\lfloor k/(\tau+1)\rfloor}$, we have that for any $i\in\mc{V}$,
\begin{equation*}
    \|x_i^k-x_i^\star\|\le V^{\lfloor k/(\tau+1)\rfloor}\overset{\eqref{eq:Vt_linear}}{\le} \rho^{\lfloor k/(\tau+1)\rfloor}V^0,
\end{equation*}
which, together with $V^0\le \|\bx^0-\bx^\star\|_{b,\infty}^2$, indicates \eqref{eq:conv_linear}.

\subsubsection{Convergence under total asynchrony (Assumption \ref{asm:totalasynchrony})}\label{sssec:total} Under total asynchrony, the first step of \eqref{eq:uik1_4} still holds:
\begin{equation}\label{eq:xikxistar_total}
    \|x_i^k-x_i^\star\|^2\le  \rho\sum_{j\in\mc{V}}\hat{w}_{ij}\|x_j^{s_{ij}^{t_i^{k}}}-x_j^\star\|^2.
\end{equation}
Define $\{a_t\}$ as: $a_0=0$; For each $t\in\N_0$, $a_{t+1}$ is the smallest integer such that for all $k\ge a_{t+1}$, if $k\in\mc{K}_i$ for some $i\in\mc{V}$, then
\begin{equation}\label{eq:sij_range_total}
    s_{ij}^{t_i^k}\ge a_t,~\forall j\in\bar{\mc{N}}_i.
\end{equation}
By Assumption \ref{asm:totalasynchrony}, the sequence $\{a_t\}$ is well defined. Moreover, by \eqref{eq:xikxistar_total} and \eqref{eq:sij_range_total}, we have
\begin{equation*}
    \max_{k\in [a_{t+1}, a_{t+2})}\|\bx^k-\bx^\star\|_{b,\infty}^2\le  \rho\max_{k\in [a_{t}, a_{t+1})} \|\bx^k-\bx^\star\|_{b,\infty}^2.
\end{equation*}
Therefore,
\begin{equation*}
    \lim_{t\rightarrow+\infty} \max_{k\in [a_{t+1}, a_{t+2})}\|\bx^k-\bx^\star\|_{b,\infty}^2 = 0,
\end{equation*}
which indicates $\lim_{k\rightarrow+\infty} \|\bx^k-\bx^\star\|_{b,\infty}^2=0$.

\subsection{Proof of Theorem \ref{thm:epsilon_0_adapt}}\label{proof:ada_epsilon_0}

Most derivations in Appendix \ref{proof:epsilon_0} holds when replacing $(\beta_i,\gamma_i)$ with $(\beta_i^k,\gamma_i^k)$. When using adaptive step-size $\gamma_i^k$, $\hat{w}_{ij}$ in Appendix \ref{proof:epsilon_0} will become
\begin{equation}\label{eq:hat_wij_def_adapt}
    \hat{w}_{ij}^k = \begin{cases}
        w_{ij}\gamma_i^k/\alpha, & j\ne i,\\
        1-(1-w_{ii})\gamma_i^k/\alpha, & \text{otherwise}.
    \end{cases}
\end{equation}

To guarantee the results in Lemma \ref{lemma:smooth_deter} holds when replacing $\beta_i,\gamma_i,\hat{w}_{ij}$ with $\beta_i^k,\gamma_i^k,\hat{w}_{ij}^k$, we only need to guarantee \eqref{eq:beta_ik} that is already assumed and
\begin{align}
    \gamma_i^k &\le \frac{\hat{w}_{ii}^k}{\beta_i^k},\label{eq:gammaik_lemma}\\
    \hat{w}_{ij}^k &>0~\forall j\in\bar{\mc{N}}_i.\label{eq:hatwij>0}
\end{align}
By \eqref{eq:hat_wij_def_adapt}, we have $\hat{w}_{ii}^k = 1-(1-w_{ii})\gamma_i^k/\alpha$, so that \eqref{eq:gammaik_lemma} is equivalent to
\begin{equation}\label{eq:suffi_cond}
    \gamma_i^k\left((1-w_{ii})/\alpha+\beta_i^k\right) \le 1,
\end{equation}
which naturally holds since the left-hand side equals to $\eta_i<1$ by \eqref{eq:ada_step-cond}. As a result, \eqref{eq:gammaik_lemma} hold. 

By \eqref{eq:ada_step-cond_lowerbound}, $\gamma_i^k>0$, so that $\hat{w}_{ij}^k>0$ $\forall j\in\mc{N}_i$. Now, to show \eqref{eq:hatwij>0}, it sufficies to show $\hat{w}_{ii}^k>0$, which is equivalent to $\gamma_i^k\le \frac{1-w_{ii}}{\alpha}$. By \eqref{eq:ada_step-cond} and $\eta_i<1$,
\begin{equation}\label{eq:gamma_i_bound}
    \gamma_i^k \le \frac{\eta_i}{\beta_i^k+\frac{1-w_{ii}}{\alpha}}\le \frac{\alpha\eta_i}{1-w_{ii}} \le \frac{\alpha}{1-w_{ii}} ,
\end{equation}
so that $\hat{w}_{ii}^k>0$.
Consequently, Lemma \ref{lemma:smooth_deter} holds when replacing $\beta_i,\gamma_i,\hat{w}_{ij}$ with $\beta_i^k,\gamma_i^k,\hat{w}_{ij}^k$.

To guarantee the results in {\bf Steps 1-5} and Appendix \ref{sssec:partial}), we only need 
\begin{equation}\label{eq:hatw_ii_positive}
\min_{i\in\mc{V}}\min_{k\in\mc{K}_i}\min_{j\in\bar{\mc{N}}_i} \hat{w}_{ij}^k>0
\end{equation}
and \eqref{eq:gammaik_lemma} that has been proved. To show \eqref{eq:hatw_ii_positive}, note that
\[\min_{i\in\mc{V}}\min_{k\in\mc{K}_i}\min_{j\in \mc{N}_i} \hat{w}_{ij}^k\ge \frac{\min_{i\in\mc{V}}\min_{j\in \mc{N}_i} w_{ij}}{\alpha}\min_{k\in\mc{K}_i}\gamma_i^k,\]
where the right-hand side is positive due to \eqref{eq:ada_step-cond_lowerbound}. Moreover, by \eqref{eq:gamma_i_bound}, $\hat{w}_{ii}^k\ge 1-\eta_i$. Therefore, \eqref{eq:hatw_ii_positive} holds.

Concluding all the above, the results in Theorem \ref{thm:epsilon_0_adapt} hold.

\subsection{Proof of Theorem \ref{thm:sto_smooth}}\label{proof:sto_smooth}

Note that the equations \eqref{eq:xi_optexp} and \eqref{eq:cross_term} still hold even in the stochastic setting. For all $i\in\mc{V}$ and $k\in\mc{K}_i$, due to the convexity of $m_i^k$ and $h_i$,
\begin{equation}\label{eq:-gammaigik1}
    \begin{split}
        &-\gamma_i \langle g_i^{k+1}, x_i^{k+1}-x_i^\star\rangle\\
        \le &\gamma_i((m_i^k+h_i)(x_i^\star)-(m_i^k+h_i)(x_i^{k+1}))\\
        = & \gamma_i(\phi_i(x_i^\star)-\phi_i(x_i^{k+1}))+\underbrace{\gamma_i(m_i^k(x_i^\star)-f_i(x_i^\star))}_{\Delta_1}\\
        &+\underbrace{\gamma_i(f_i(x_i^{k+1}) - m_i^k(x_i^{k+1}))}_{\Delta_2}.
    \end{split}
\end{equation}
Substituting \eqref{eq:-gammaigik1} and \eqref{eq:cross_term} into \eqref{eq:xi_optexp} gives
    \begin{equation*}
        \begin{split}
        \|x_i^{k+1}-x_i^\star\|^2&\le \sum_{j\in \bar{\mc{N}}_i}\hat{w}_{ij}\|x_j^{s_{ij}^k}-x_j^\star\|^2\\
            &+2(\Delta_1+\Delta_2-\frac{\hat{w}_{ii}\|x_i^{k+1}-x_i^k\|^2}{2})\\
            &-2\gamma_i(\phi_i(x_i^{k+1})-\phi_i(x_i^\star)-\langle g_i^\star, x_i^\star-x_i^{k+1}\rangle),
    \end{split}
    \end{equation*}
    which, together with the strong convexity of $\phi_i$, yields
    \begin{equation}\label{eq:uik1_sto}
    \begin{split}
        \|x_i^{k+1}-x_i^\star\|^2&\le \rho\sum_{j\in \bar{\mc{N}}_i}\hat{w}_{ij}\|x_j^{s_{ij}^k}-x_j^\star\|^2\\
            &+2\rho(\Delta_1+\Delta_2-\frac{\hat{w}_{ii}\|x_i^{k+1}-x_i^k\|^2}{2}).
    \end{split}
    \end{equation}
    Moreover, taking expectations on $\Delta_1,\Delta_2$ over $\mc{F}^k$ and using Assumption \ref{asm:stoch_tilde_f} and \eqref{eq:gamma_i_smaller} yield
    \begin{align*}
        &\mbb{E}[\Delta_1|~\mc{F}^k] \le \epsilon_i^\star,\\
        &\mbb{E}\left[\Delta_2-\frac{\hat{w}_{ii}\|x_i^{k+1}-x_i^k\|^2}{2}\big|~\mc{F}^k\right] \le \epsilon_i.
    \end{align*}
    Taking expectations on both sides of \eqref{eq:uik1_sto} over $\mc{F}^k$ and by Assumption \ref{asm:partialasynchrony} and the above equation, we have
    \begin{equation}\label{eq:ind_exp_xik1}
    \begin{split}
        &\mbb{E}[\|x_i^{k+1}-x_i^\star\|^2|~\mc{F}^k]\\
        \le& \rho \sum_{j\in \bar{\mc{N}}_i}\hat{w}_{ij}\mbb{E}[\|x_j^{s_{ij}^k}-x_j^\star\|^2|~\mc{F}^k]+2\rho(\epsilon_i+\epsilon_i^\star)\\
        =& \rho \sum_{j\in \bar{\mc{N}}_i}\hat{w}_{ij}\mbb{E}[\|x_j^{s_{ij}^k}-x_j^\star\|^2|~\mc{F}^{s_{ij}^k-1}]+2\rho(\epsilon_i+\epsilon_i^\star).
    \end{split}
    \end{equation}
    Under Assumption \ref{asm:partialasynchrony}, \begin{equation}\label{eq:ind_exp_xik1_partial}
    \begin{split}
        &\mbb{E}[\|x_i^{k+1}-x_i^\star\|^2|~\mc{F}^k]\\
        \le& \rho \max_{k-D\le t\le k}\max_{j\in\mc{V}}\mbb{E}[\|x_j^t-x_j^\star\|^2|~\mc{F}^{t-1}]+2\rho(\epsilon_i+\epsilon_i^\star).
    \end{split}
    \end{equation}
    By the above equation, we will show that for all $k\in\mc{I}^t$ where $\mc{I}^t$ is defined in \eqref{eq:It_def}, it holds that
    \begin{equation}\label{eq:ind_res}
        \max_{i\in\mc{V}}\mbb{E}[\|x_i^k-x_i^\star\||~\mc{F}^{k-1}]\le \rho^t\max_{i\in\mc{V}}\|x_i^0-x_i^\star\|^2+C.
    \end{equation}
    By \eqref{eq:ind_exp_xik1_partial}, it is straightforward to see \eqref{eq:ind_res} at $t=0$. Suppose that \eqref{eq:ind_res} holds for some $t\in \N_0$. Then, for each $k\in\mc{I}^{t+1}$, 
    \begin{equation*}
        \begin{split}
            &\mbb{E}[\|x_i^k-x_i^\star\|^2|~\mc{F}^{k-1}]\\
            =&\mbb{E}[\|x_i^{t_i^k+1}-x_i^\star\|^2|~\mc{F}^{t_i^k}]\\
            \le& \rho \max_{t_i^k-D\le t\le t_i^k}\max_{j\in\mc{V}}\mbb{E}[\|x_j^t-x_j^\star\|^2|~\mc{F}^{t-1}]+2\rho(\epsilon_i+\epsilon_i^\star)\\
            \le& \rho^{t+1} \max_{i\in\mc{V}}\|x_i^0-x_i^\star\|^2+C.
        \end{split}
    \end{equation*}
    Therefore, \eqref{eq:ind_res} holds for all $t\in\N_0$, which implies \eqref{eq:sto_conv}.

    For the convergence under total asynchrony, by \eqref{eq:ind_exp_xik1} and following the proof in Section \ref{sssec:total}, we obtain \eqref{eq:lim_sup}.

\section*{Reference}
\bibliography{reference.bib}
\bibliographystyle{unsrt}

\begin{IEEEbiography}[{\includegraphics[width=1in,height=1.25in,clip]{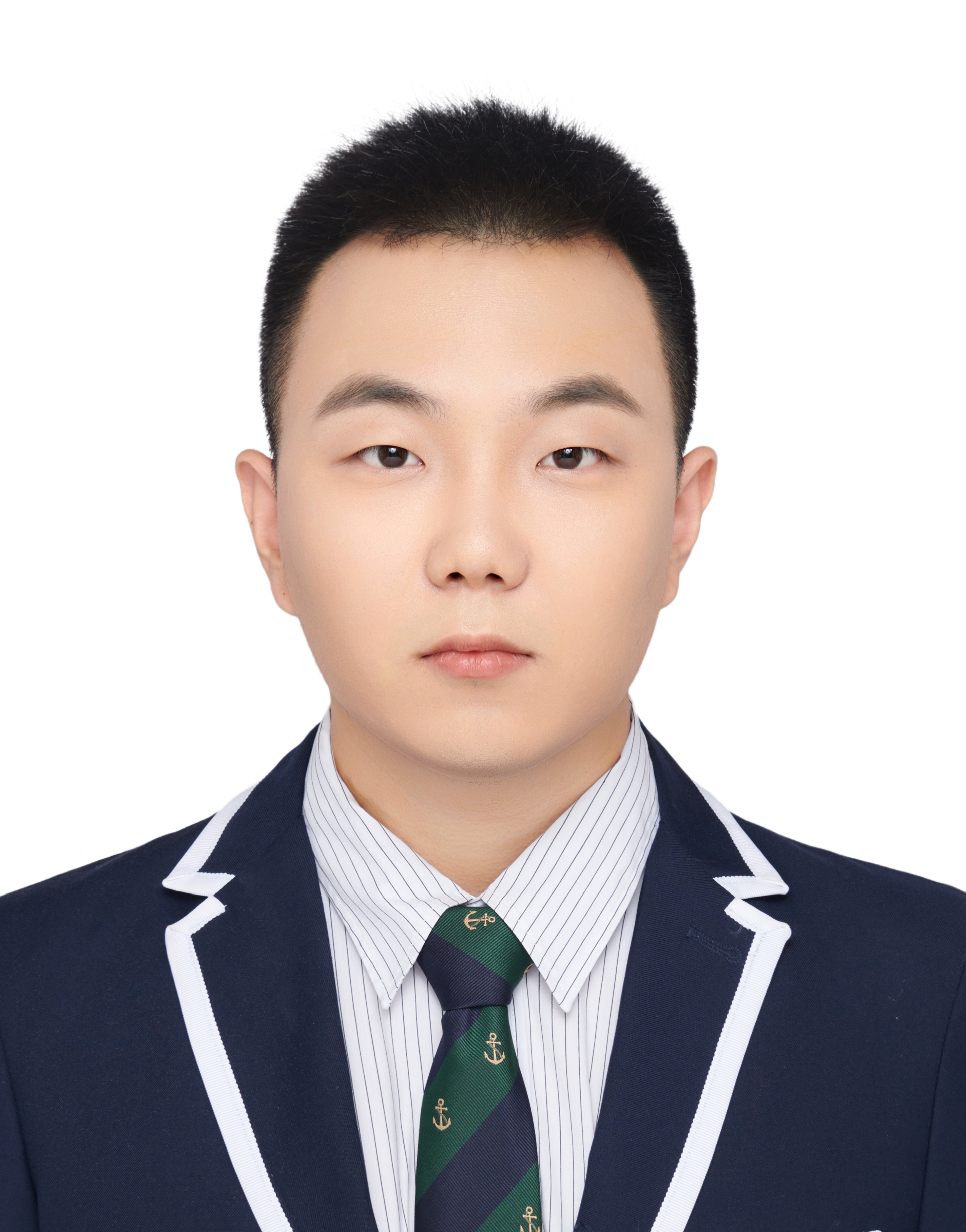}}]{Zhao Zhu } received the B.S. degree in Mathematics and Applied Mathematics from Hainan University, Hainan, China, in 2025. He is currently pursuing the M.S. degree in Control Science and Engineering at Southern University of Science and Technology, Shenzhen, China. His research interests include distributed optimization and asynchronous optimization.
\end{IEEEbiography}

\begin{IEEEbiography}[{\includegraphics[width=1in,height=1.25in,clip,keepaspectratio]{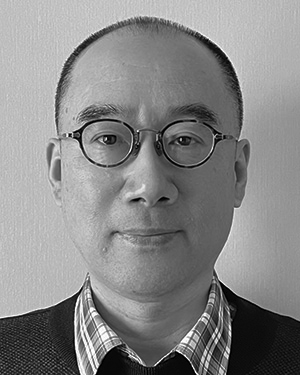}}]{Yu-Ping Tian}
received the bachelor’s degree in automation from Tsinghua University, Beijing, China, in 1986, the Ph.D. degree in electrical engineering from Moscow Power Institute, Moscow, USSR, in 1991, and the Sc.D. degree in electrical engineering from Taganrog State Radio-engineering University, Taganrog, Russia, in 1996.

From January 1992 to December 2018, he was a Faculty with the School of Automation, Southeast University, Nanjing, China. From December 2018 to January 2024, he was a Professor with the School of Automation, Hangzhou Dianzi University, Hangzhou, China. From January 2024 to July 2024, he was a Professor with the School of System Design and Intelligent Manufacturing, Southern University of Since and Technology, Shenzhen, China. Since July 2024, he has been with Hangzhou City University, Hangzhou, where he is a Professor with the School of Information and Electrical Engineering. He has also held visiting appointments with several other universities, including Central Queensland University, Rockhampton, QLD, Australia; University of California at Berkeley, Berkeley, CA, USA; and City University of Hong Kong, Hong Kong. His research interests include the control and estimation theory for engineering systems, especially distributed control and estimation theory for sensor networks, multirobot systems, and cyber-physical systems.

Prof. Tian is a recipient of the Chang Jiang Distinguished Professorship awarded by the Education Ministry of China and the Distinguished Young Scholar Award of the National Natural Science Foundation of China.
\end{IEEEbiography}

\begin{IEEEbiography}[{\includegraphics[width=1in,height=1.25in,clip]{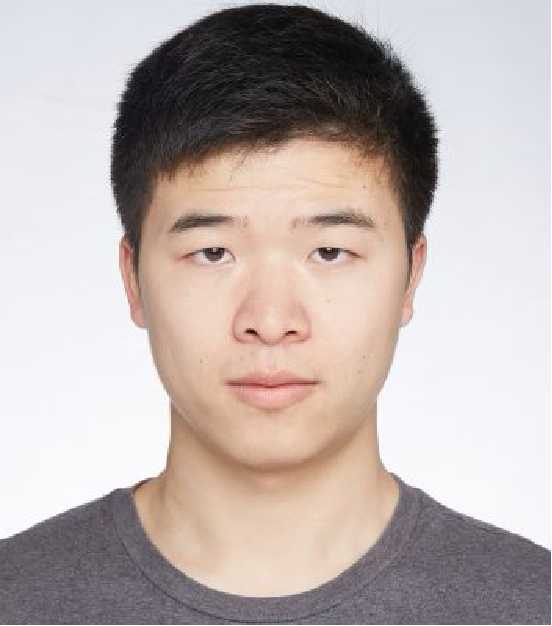}}]{Xuyang Wu} (Member, IEEE) received the B.S. degree in Information and Computing Science from Northwestern Polytechnical University, Xi'an, China in 2015, and the Ph.D. degree from the University of Chinese Academy of Sciences, China in 2020. He was a postdoctoral researcher in the Division of Decision and Control Systems at KTH Royal Institute of Technology, Stockholm, Sweden, from 2020 to 2023. He is currently an Associate Professor in the School of Automation and Intelligent Manufacturing at the Southern University of Science and Technology, Shenzhen, China. His research interests include distributed optimization and machine learning.
\end{IEEEbiography}

\end{document}